\documentclass{article}

\usepackage{arxiv}
\usepackage{lineno,hyperref}
\usepackage[utf8]{inputenc}
\usepackage{amsthm, amsmath, amsfonts, amssymb}
\usepackage[fleqn]{mathtools} 
\DeclarePairedDelimiter\innerproduct{\langle}{\rangle}
\usepackage{csquotes}    
\usepackage{mathrsfs}
\usepackage{xcolor}
\usepackage{verbatim}
\DeclareMathOperator*{\esssup}{ess\,sup}
\newtheorem{lem}{Lemma}[section]
\newtheorem{thm}{Theorem}[section]

\newtheorem{remark}{Remark}

\newtheorem{definition}{Definition}[section]
\newcommand{\mycomment}[1]{}

\newcommand{\br}[2]{\left(#1 , #2\right)}

\newcommand{\nr}[1]{\| #1 \|_{L^2}^2}
\newcommand{\nrm}[1]{\| #1 \|_{L^2}}
\newcommand{\norm}[1]{{\left\lVert#1\right\rVert}_{L^{\infty}(\Omega)}}
\usepackage{multicol}
\usepackage{graphicx}
\usepackage{epstopdf}
\usepackage{caption}
\usepackage{subcaption}
\usepackage{adjustbox}
\usepackage{subcaption}
\usepackage{tikz}
\usepackage{pgffor}

\title{Global Existence and Finite-Time Blow-Up for a Coupled Darcy-Forchheimer-Brinkman System with Quadratic Reaction Dynamics }

\author{
Sahil Kundu$^{1}$\thanks{Email: sahil.kundu.in@gmail.com},
Manmohan Vashisth$^{1}$\thanks{Email: manmohanvashisth@iitrpr.ac.in },
Manoranjan Mishra$^{1}$\thanks{Email: manoranjan@iitrpr.ac.in}\\[0.5em]
$^{1}$Department of Mathematics, Indian Institute of Technology Ropar, Rupnagar, India
}

\date{ 24 May, 2024}

\renewcommand{\undertitle}{}

\begin{document}
\maketitle

\begin{abstract}
	We study a nonlinear system coupling the Darcy-Forchheimer-Brinkman equations with a convection-diffusion-reaction equation, arising in reactive transport through porous media. The model features a nonlinear viscosity coupling, Forchheimer inertial drag, convective transport, and a quadratic reaction term. We establish the existence of local-in-time weak solutions for general initial data. Under the physically relevant condition on initial data $0 \leq c_0 \leq 1$, a maximum principle for the concentration is proved, yielding global existence and uniqueness of weak solutions in two and three space dimensions. For higher regular initial data, we obtain the existence, uniqueness, and continuous dependence of strong solutions. In this regime, the concentration decays exponentially to zero in $L^p$-norm for all $1 \leq p \leq \infty$ with a uniform decay rate. In contrast, if $c_0 > 1$, we demonstrate the occurrence of finite-time blow-up of solutions and derive an explicit upper bound for the blow-up time. Finally, numerical simulations based on the finite element method are presented to illustrate both the decay behavior and finite-time blow-up predicted by the theory.
\end{abstract}


\section{Introduction}
 The study of coupled fluid dynamics and reactive transport in porous media is essential for modeling numerous phenomena across engineering, environmental sciences and biophysics. This coupling arises in a wide range of physical applications, including enhanced oil recovery \cite{lake2014fundamentals}, groundwater filtration \cite{dentz2011mixing}, and the geological sequestration of carbon dioxide \cite{huppert2014fluid, emami2015convective}. In slow viscous flows, the velocity of fluids can be modeled by the classical Darcy law, which fails to capture the behavior observed at higher flow rates where inertial effects become significant \cite{zeng2006criterion,ma1993microscopic}. Therefore, a non-linear correction to the Darcy equation is necessary to accurately model these regimes.

 This limitation was addressed by Forchheimer in 1901 \cite{Forchheimer1901}, who modified Darcy's law by introducing a quadratic velocity term, known as the Forchheimer term, in the momentum balance equation. After combining this with the Brinkman extension \cite{brinkman1949calculation}, incorporating viscous diffusion, which results into the  Darcy-Forchheimer-Brinkman (DFB) equation. This unified model effectively accounts for linear Darcy resistance, nonlinear Forchheimer inertia, and macroscopic viscous shear effects. The unsteady DFB equations for fluid motion are given by
\begin{align*}
\frac{\partial \boldsymbol{u}}{\partial t} + \frac{\mu(c)}{K} \boldsymbol{u} + \beta |\boldsymbol{u}| \boldsymbol{u} = -\boldsymbol{\nabla} p + \mu_{e} \Delta \boldsymbol{u} + \boldsymbol{f},
\end{align*}
coupled with the incompressibility condition:
\begin{equation*}
\nabla \cdot \boldsymbol{u} = 0.
\end{equation*}
Here, $\boldsymbol{u}$ and $p$ represent the velocity and pressure, respectively. The parameters $K$ and $\beta$ denote the permeability and the Forchheimer coefficient of the porous medium. Furthermore, $\mu(c)$ is the concentration-dependent viscosity, $\mu_e$ is the effective viscosity and $\boldsymbol{f}$ denotes an external force. The flow dynamics are coupled with a convection-diffusion-reaction equation describing the transport of the solute concentration $c$:
\begin{align*} 
\frac{\partial c}{\partial t} + \boldsymbol{u}\cdot \boldsymbol{\nabla}c = D \Delta c + R(c),
\end{align*}
where $D$ is the molecular diffusion coefficient and $R(c)$ represents the reaction source term. In this work, we focus on reaction kinetics of the quadratic type, specifically $R(c) = \kappa\,c(c-1)$. This non-linearity is physically significant, serving as a prototype for autocatalytic reactions \cite{fisher1937wave} and combustion processes \cite{williams2018combustion}. However, such reactions introduce substantial analytical challenges, as the superlinear term $c(c-1)$ becomes destabilizing once the concentration exceeds the unit threshold $c=1$. Unlike logistic-type nonlinearities, this term precludes the direct use of standard maximum principles to obtain a global $L^\infty$ bound (see \cite{EvansPDE}). To address this, we establish a maximum principle for concentration via a contradiction argument, testing the weak formulation with carefully constructed cutoff functions, which amounts to a uniform bound for concentration. This uniform bound is crucial to ensure global existence and uniqueness of the solution. We assume that the viscosity function $\mu$ is locally Lipschitz continuous and non-negative preserving, i.e., $\mu(s)\geq 0$ whenever $s\geq 0$.
Despite a strong non-linear coupling, we establish existence and uniqueness of solutions under the above-mentioned assumption on viscosity, which are weaker than those of prior works where typical assumptions of global Lipschitz continuity or uniform boundedness of the viscosity are taken (see, for instance, \cite{Amirat1998, Feng1995,roy2025existence} and therein).
We further emphasize that our analysis is applicable to a wide range of standard models, including those with exponential or polynomial viscosity profiles, which are physically relevant and have been extensively investigated in numerical studies (see, e.g., \cite{TanHomsy1986, Manickam1993, Pramanik_Hota_Mishra_2015}).

The mathematical analysis of fluid flow in porous media is well-established, particularly when the hydrodynamic equations are considered in isolation from transport phenomena. For the stationary Darcy-Brinkman-Forchheimer equations, the existence and uniqueness of weak solutions have been established in both bounded and unbounded domains (see, e.g., \cite{Varsakelis_2017}). Regarding the unsteady case, global existence and continuous dependence results for equations with constant coefficients were derived in \cite{ugurlu2008existence}. Building on these results, Hajduk and Robinson \cite{HAJDUK20177141} proved the existence of global-in-time smooth solutions for the convective Brinkman–Forchheimer equations under specific assumptions on the viscosity and Forchheimer coefficients. Furthermore, Medková \cite{MEDKOVA2024398} extended this analysis to the Darcy-Forchheimer-Brinkman system in bounded domains, utilizing a fixed-point theorem to establish the existence of solutions for the associated mixed problem.

However, results for fully coupled systems involving concentration-dependent physical properties, such as viscosity and density, remain limited. Sayah et al. \cite{refId0} established existence and uniqueness of a stationary convection–diffusion–reaction system coupled with the Darcy–Forchheimer equation using a Galerkin approach, under the assumptions of constant viscosity and density and with a linear reaction term. In the unsteady setting, well-posedness results for fully coupled models with constant coefficients were obtained by Titi and Trabelsi \cite{titi2024global}, who proved global well-posedness for a three-dimensional Brinkman–Forchheimer–Bénard convection problem, while Kundu et al. \cite{kundu2024existence} analyzed an unsteady Darcy–Brinkman system for miscible flows allowing concentration-dependent permeability but still assuming constant viscosity. More recently, Roy and Pramanik \cite{roy2025existence} investigated an unsteady Darcy–Brinkman system with quadratic chemical reactions. Although global well-posedness was established in the non-reactive case, the presence of quadratic reaction kinetics, leading to superlinear growth, restricted the analysis to local-in-time existence. Consequently, to the best of our knowledge, global existence and uniqueness of a fully coupled Darcy–Forchheimer–Brinkman system interacting with a convection–diffusion–reaction equation featuring quadratic reaction kinetics remain an open problem, which is one of the main objectives of this paper. 

We provide a comprehensive analytical framework for this coupled system in two and three spatial dimensions ($d=2, 3$). More precisely, using the Faedo–Galerkin method, combined with truncation techniques and a priori energy estimates, we establish the existence of a local solution. Crucially, under the assumption that the initial concentration satisfies $0 \leq c_0 \leq 1$, we prove that the established local in time solution can be extended to the whole time interval. Furthermore, if we have additional regularity on the initial data, we prove the existence of a strong solution and its continuous dependence on initial conditions. Finally, we investigate the asymptotic behavior and stability thresholds for the above-mentioned system. We prove that the concentration decays exponentially when $0\leq c_0\leq 1$, while there is a finite-time blow up in the concentration when $c_0 > 1$. This shows that the assumption $0\leq c_0\leq 1$ is crucial for the global well-posedness of the system given by \eqref{model1}-\eqref{assumption:beta-k} in section \ref{sec: mathematical model} below. Finally, we numerically simulate the above-mentioned system using the Finite Element Method in COMSOL Multiphysics and verify our theoretical findings. 

The remainder of this paper is organized as follows. In Sections \ref{sec: mathematical model} and \ref{sec: notations}, we introduce the mathematical model, relevant function spaces, and key inequalities. Section \ref{sec: Existence} is devoted to the global well-posedness theory, where we prove the existence of local solutions via the Faedo-Galerkin method and further extend them globally using the maximum-minimum principles. In Section \ref{sec: strong solution}, we prove the higher order regularity of these solutions, which results in the existence of strong solutions.  Section \ref{sec: Asymptotic Behaviour} investigates the long-time behavior, presenting a dichotomy between exponential decay for sub-threshold solutions and finite-time blow-up for super-threshold data. In section \ref{Numerical Validation}, a detailed numerical investigation based on a finite element discretization is provided, which corroborates our theoretical predictions regarding the decay behavior and the formation of singularities. Finally, we offer concluding remarks and discuss potential directions for future research in \ref{sec: conclusion}.

\section{ Mathematical Model }\label{sec: mathematical model}
We consider the miscible displacement of a fluid in a saturated porous medium, accounting for the effects of density-driven flow, concentration-dependent viscosity and linear adsorption. Let $\Omega \subset \mathbb{R}^d$ ($d=2$ or $3$) be a bounded domain representing the porous medium and for $T>0$,  $[0, T]$ denote the time interval of interest.
 The following system of equations governs the dynamics of the flow:
\begin{subequations}\label{model}
\begin{equation}\label{model1}
  \boldsymbol{\nabla} \cdot \boldsymbol{u} = \boldsymbol{0}  \quad \text{in}\ (0,T) \times \Omega,  
\end{equation}
\begin{equation}\label{model2}
\frac{\partial \boldsymbol{u}}{\partial t} + \frac{\mu(c)}{K}\boldsymbol{u} + \beta |\boldsymbol{u}|\boldsymbol{u} = - \boldsymbol{\nabla} p  + \mu_{e}\Delta \boldsymbol{u} + \boldsymbol{f}  \quad \text{in}\ (0,T) \times \Omega,
\end{equation}
\begin{equation}\label{model3}
    \frac{\partial c}{\partial t} + \boldsymbol{u} \cdot \boldsymbol{\nabla} c = D \Delta c + \kappa\, c(c-1) \quad \text{in}\ (0,T) \times \Omega 
\end{equation} where ${\bf u},c,\beta,p,\mu_e,K, D$ and $\kappa$ are as defined in section  \ref{Sec: Introduction}. 
Equation \eqref{model1} expresses the conservation of mass for an incompressible fluid,  equation \eqref{model2} corresponds to the conservation of momentum, and equation \eqref{model3} represents the convection diffusion reaction. 
The system is supplemented with no-flux boundary conditions
\begin{equation}\label{boundary conditions}
\boldsymbol{u} = \boldsymbol{0}, \quad \boldsymbol{\nabla}c \cdot \boldsymbol{\nu} = 0 \quad \text{on } (0,T) \times \partial \Omega,
\end{equation}
and the initial concentration distribution is prescribed by
\begin{equation}\label{initial conditions}
\boldsymbol{u}(0, \boldsymbol{x}) = \boldsymbol{u}_0(\boldsymbol{x}), \quad c(0,\boldsymbol{x}) = c_{0}(\boldsymbol{x}) \quad \text{for } \boldsymbol{x} \in \Omega.
\end{equation}
We further assume that the Forchheimer coefficient $\beta \in L^{\infty}(\Omega)$ and  the reaction rate  $\kappa\in L^\infty(\Omega)$  satisfy 
\begin{align}\label{assumption:beta-k}
0  \leq \beta(\boldsymbol{x}) \leq \beta_2 \quad \text{and} \quad 0 < \kappa_1 \leq \kappa(\boldsymbol{x}) \leq \kappa_2,\ \mbox{for  a.e. $\boldsymbol{x} \in \Omega$ }  
\end{align}
and some 
constants $ \beta_2, \kappa_1$ and  $\kappa_2$. We also assume that the viscosity function $\mu$ is locally Lipschitz and preserves non-negativity (i.e., $\mu(s) \geq 0$ for $s \geq 0$). 
\end{subequations}


\section{ Preliminaries and Concept of Weak Solution}\label{sec: notations}
In this section, we introduce the standard notations and recall certain well-known analytical results that will be used throughout the paper. Let $\Omega \subset \mathbb{R}^d$ (with $d = 2$ or $3$) be an open, bounded domain with a $C^1$ boundary $\partial \Omega$. The Sobolev spaces $W^{k,p}(\Omega)$ and $H^k(\Omega)$ are understood in the usual sense, endowed with their standard norms. The notations $\br{\cdot}{\cdot}$ and $\nrm{\cdot}$ denote the standard $L^2$-inner product and $L^2$-norm respectively, in both $L^2(\Omega)$ and $(L^2(\Omega))^d$ which can be understood in the given context.
The solution spaces for the velocity field are defined by 
\[
\boldsymbol{S} := \left\{ \boldsymbol{v} \in (L^{2}(\Omega))^d \  : \; \nabla \cdot \boldsymbol{v} = 0 \text{ in } \Omega \right\} \ \mbox{and} \ \boldsymbol{V} := \left\{ \boldsymbol{v} \in (H^{1}_{0}(\Omega))^d \ :\  \nabla \cdot \boldsymbol{v} = 0 \text{ in } \Omega \right\}.
\] 
Throughout this paper, $\langle \cdot, \cdot \rangle$ will denote the duality pairing between a Banach space and its dual. The specific spaces involved (e.g., between $V^*$ and $V$ or $(H^1)^*$ and $H^1$) are understood by the context. For any Banach space $Y$ with norm $\lVert\cdot\rVert_Y$, we define the Bochner spaces
$$ {L^{p}(a, b ; Y)} := \left\{ \phi:[a, b] \rightarrow Y : \ \mbox{$\phi$ is strongly measurable and } \|\phi\|_{L^{p}(a, b ; Y)}:=\left(\int_{a}^{b}\|\phi\|_{Y}^{p} \, \mathrm{d} t \right)^{\frac{1}{p}}  < \infty      \right\} ~ ~ \text{for} ~~1 \leq p < \infty, $$ 
and 
$${L^{\infty}(a, b ; Y)} := \left\{ \phi:[a, b] \rightarrow Y : \ \mbox{$\phi$ is strongly measurable and } \|\phi\|_{L^{\infty}(a, b ; Y)}:= \esssup_{[a, b]}\|\phi\|_{Y} < \infty      \right\}.$$
We denote by $ M\in [0, \infty)$ a generic constant which will be context-dependent.
\begin{thm}[Gagliardo–Nirenberg, cf. { \cite[Lemma~1]{migorski2019nonmonotone}}, { \cite[Lemma~1.1]{garcke2019}}]\label{gagliardo} 
If $\Omega \subset \mathbb{R}^d \,(d=2,3)$ is a domain with $\mathcal{C}^{1}$ boundary, then for any $ \phi \in H^{1}(\Omega)$ there exists a constant $M>0$ depending only on $\Omega$ such that the following inequality holds when $d=2$ 
$$
\|\phi\|_{L^{4}} \leq M\|\phi\|_{L^{2}}^{1 / 2}\| \phi\|_{H^{1}}^{1 / 2},   
$$
and for $d=3$, the inequality 
$$
\|\phi\|_{L^{4}} \leq M\|\phi\|_{L^{2}}^{1 / 4}\| \phi\|_{H^{1}}^{3 / 4} \hspace{5pt}\hspace{5pt} \mbox{holds}. 
$$  
\end{thm}

\begin{thm}[Aubin-Lions, cf. {\cite[Lemma~3]{migorski2019nonmonotone}} ]\label{aubin-lions}
Let $X_{1}, X_{2}$ and $X_{3}$ be reflexive Banach spaces.  If the embeddings $X_1\hookrightarrow X_2$ and $X_2\hookrightarrow X_3$ are compact and continuous respectively, then for any $p,q\in (1,\infty)$ and $T>0$ we have the embedding 
$$
\left\{\phi \in L^{p}(0, T; X_{1}): \dfrac{\partial \phi}{\partial t} \in L^{q}(0, T; X_{3})\right\}\hookrightarrow  L^{p}(0, T; X_{2})$$ 
is compact where $\displaystyle \frac{\partial \phi}{\partial t}$ denotes the weak partial derivative with respect to the time variable. Furthermore, the embedding   \[\left\{\phi \in L^{p}(0, T ; X_{1}) : \dfrac{\partial \phi}{\partial t} \in L^{q}(0, T ; X_{3})\right\}\hookrightarrow C([0, T] ; X_{2}) \] is also compact provided $p=\infty$ and  $q>1$ is arbitrary. 
\end{thm}
\begin{thm}[Young's inequality with $\epsilon$] \label{youngs inequaity}
  Let $p,q$ be positive real numbers satisfying $\frac{1}{p} + \frac{1}{q} = 1$, and $a,b$ are non-negative real numbers. Then for any given $\epsilon  > 0 $, there exist a non-negative constant say $M_{\epsilon}$ such that \[a b \leq \epsilon a^p + M_{\epsilon} b^q.\]  
\end{thm} 

\begin{thm}[Comparison principle for ODEs, cf. {\cite{lakshmikantham1969differential}}] \label{Comparison principle}
Let $f:[0, T]\times\mathbb{R}\to\mathbb{R}$ be continuous and nondecreasing in its second argument.
Assume that $y_1,y_2:[0,T]\to\mathbb{R}$ are absolutely continuous functions such that
\[
y_1'(t)\le f(t,y_1(t)), \qquad
y_2'(t)=f(t,y_2(t)) \quad \text{for almost every } t\in(0,T),
\]
with
\[
y_1(0)\le y_2(0).
\]
Then
\[
y_1(t)\le y_2(t)\quad \text{for all } t\in[0,T].
\]
\end{thm}

Below, we discuss the concept of weak and strong solutions of the system \eqref{model}]\label{def: weak solution}.
\begin{definition}[Weak and Strong Solutions of System~\eqref{model}]\label{def: weak solution}
We say that a pair \( (c, \boldsymbol{u}) \)
satisfying \[c \in L^\infty(0,T;L^{2}(\Omega)) \cap L^2(0,T;H^1(\Omega)), \  \frac{\partial c}{\partial t} \in L^2(0,T;(H^1(\Omega))^{\ast}), \ \boldsymbol{u} \in L^{\infty}(0,T;\boldsymbol{S}) \cap L^2(0,T; \boldsymbol{V}) \ \mbox{and} \  \frac{\partial \boldsymbol{u}}{\partial t} \in L^2(0,T;\boldsymbol{V}^{\ast})\] 
is a weak solution for the system of initial boundary value problems (IBVP)    \eqref{model1}–\eqref{initial conditions} provided  $c$ and $\boldsymbol{u}$ satisfy the initial conditions \eqref{initial conditions}  and for all $({\bf v},\phi)\in {\bf V}\times H^1(\Omega)$ the following variational  identities 
 \begin{align}\label{weak1}
   \innerproduct{\frac{\partial \boldsymbol{u}(t)}{\partial t},\boldsymbol{v}} +\br{\frac{\mu(c(t))}{K}\boldsymbol{u}(t)}{\boldsymbol{v}} + \br{\beta|\boldsymbol{u}(t)|\boldsymbol{u}(t)}{\boldsymbol{v}} + \mu_{e}\br{\boldsymbol{\nabla}\boldsymbol{u}(t)}{\boldsymbol{\nabla}\boldsymbol{v }} = \br{\boldsymbol{f}}{\boldsymbol{v}} 
\end{align}
and 
\begin{align}\label{weak2}
    \innerproduct{\frac{\partial c(t)} {\partial t},\phi} - \br{c(t)\boldsymbol{u}(t)}{\boldsymbol{\nabla}\phi} + D \br{\boldsymbol{\nabla}c(t)}{\boldsymbol{\nabla}\phi} = \br{\kappa\, c(t)(c(t)-1)}{\phi} 
\end{align}
hold for almost every $t\in (0,T)$. Furthermore, a weak solution $(c, \boldsymbol{u})$ is called a strong solution if it possesses the additional regularity \[c \in  L^2(0,T; H^2(\Omega)), \  \frac{\partial c}{\partial t} \in L^2(0,T; L^2(\Omega)), \ \boldsymbol{u} \in L^2(0,T;\boldsymbol{V}\cap \boldsymbol{H}^2(\Omega)) \ \mbox{and} \  \frac{\partial \boldsymbol{u}}{\partial t} \in L^2(0,T;\boldsymbol{S}).\]    
\end{definition}

\section{Existence, Uniqueness and Continuous Dependence of Weak Solution 
}\label{sec: Existence}
\subsection{Existence of a weak solution}
This subsection is devoted to proving the existence of a weak solution of IBVP given by \eqref{model1}-\eqref{assumption:beta-k}. Under certain assumptions on initial data $(c_0,{\bf u_0})$, we prove the existence of weak solution to the aforementioned system of IBVP. More precisely, we prove the following theorem. 

\begin{thm}\label{thm:existence}
If $(c_0,\boldsymbol{u}_0) \in L^2(\Omega)\times \boldsymbol{S}$ is such that   \( 0 \leq c_0  \leq 1 \), almost everywhere in \( \Omega \), then there exists a weak solution for the system of IBVP ~\eqref{model1}--\eqref{assumption:beta-k}. 
\end{thm}
We will prove the above theorem by regularizing the viscosity by truncation process, which is given as follows.  For a fixed truncation parameter $\ell\in\mathbb{N}$, we define the truncated viscosity denoted by $\widetilde{\mu}_{\ell}$ by \[\widetilde{\mu}_{\ell}(s) := \mu\big(\max\{0, \min\{s, \ell\}\}\big),\ s\in \mathbb{R}\] 
Now, since $\mu$ is non-negative preserving and locally Lipschitz, it is bounded on the compact set $[0, \ell]$ for any $\ell\in\mathbb{N}$. This further implies that  there exists a constant $\mu_{\ell} > 0$, depending on $\ell$ such that the following holds 
\begin{equation} \label{eq:trunc_viscosity} 0 \leq \widetilde{\mu}_{\ell}(s) \leq \mu_{\ell},\ \mbox{for all $s \in \mathbb{R}$.}
\end{equation}
We first prove the existence of a sequence of weak solutions denoted by  $\{(c_\ell, \boldsymbol{u}_\ell)\}_{l\in \mathbb{N}}$ to the system of IBVP \eqref{model1}-\eqref{assumption:beta-k} when $\mu$ is replaced by $\widetilde{\mu}_{\ell}$, $\ell\in\mathbb{N}$. We say that  the system of IBVP \eqref{model1}-\eqref{assumption:beta-k} is a truncated system when $\mu=\widetilde{\mu}_{\ell}$, $\ell\in \mathbb{N}$. 
Our first aim is to prove the uniform boundedness of the sequence of weak solutions  $\{(c_\ell, \boldsymbol{u}_\ell)\}_{\ell\in\mathbb{N}}$ to the truncated system in appropriate norms and obtain a convergent subsequence of $\{(c_\ell, \boldsymbol{u}_\ell)\}_{\ell\in\mathbb{N}}$ whose limit will be the required weak solutions for the original system of IBVP \eqref{model1}-\eqref{assumption:beta-k}. 
Now, for a fixed $\ell$, using the Galerkin approximation method, we prove the existence of weak solution $(c_\ell, \boldsymbol{u}_\ell)$ to the truncated system. As  $H^1(\Omega)$ and ${\bf V}$ are separable Hilbert space, therefore we assume   $\{w_i \}_{i=1}^{\infty}$ and  $\{\boldsymbol{z}_i\}_{i=1}^{\infty}$ are orthogonal bases for $H^1(\Omega)$ and  $\boldsymbol{V}$ respectively. For $n\in\mathbb{N}$, Define  $W_n := \text{span}\{w_1,w_2,\cdots,w_n\}$ and $\boldsymbol{V}_n := \text{span}\{\boldsymbol{z}_1, \boldsymbol{z}_2,\cdots,\boldsymbol{z}_n \}$. We look for the solutions $\boldsymbol{u}_{\ell, n}:[0,T] \rightarrow \boldsymbol{V}_n$, $c_{\ell, n}:[0,T] \rightarrow W_n$, of the form 
\begin{align}\label{finite dim solution}
c_{\ell, n}(t,\boldsymbol{x}) = \sum_{j=1}^{n} \alpha_{j}(t) w_{j} ~,~ \boldsymbol{u}_{\ell, n}(t,\boldsymbol{x}) = \sum_{j=1}^{n} \beta_{j}(t)\boldsymbol{z}_j 
\end{align}
satisfying the equations 
\begin{align}\label{finite weak 1}
 \innerproduct{\dfrac{\partial \boldsymbol{u}_{\ell, n}(t)}{\partial t}, \boldsymbol{z}_k}  + \left( \frac{\widetilde{\mu}_{\ell}({c_{n, \ell}(t)})}{K}\boldsymbol{u}_{\ell, n}(t),\boldsymbol{z}_k\right) + \left(\beta |\boldsymbol{u}_{\ell, n}(t)| \boldsymbol{u}_{\ell, n}(t),  \boldsymbol{z}_k\right)+\mu_{e} \br{\boldsymbol{\nabla}\boldsymbol{u}_{\ell, n}(t)}{\boldsymbol{\nabla}\boldsymbol{z}_k} - \br{\boldsymbol{f}(t)}{\boldsymbol{z}_k}=0,    
\end{align}
\begin{align}\label{finite weak 2}
 \innerproduct{\dfrac{\partial c_{n, \ell}(t)}{\partial t},w_{k}} + \br{\boldsymbol{u}_{\ell, n}(t)\cdot \boldsymbol{\nabla}c_{\ell, n}(t)}{w_{k}} + D \br{\boldsymbol{\nabla}c_{\ell, n}(t)} {\boldsymbol{\nabla} w_{j}}  -  \br{\kappa\,c_{n, \ell}(t) (c_{\ell, n}(t) -1)}{w_{k}} =0,    
\end{align}
for $k = 1,2,3,\cdots,n$ and for almost every $t \in (0,T)$. Substituting the expressions of $c_{\ell, n}$ and  ${\bf u}_{n,\ell}$ from (\ref{finite dim solution}) into \eqref{finite weak 1} and \eqref{finite weak 2} respectively.  Finally using the  orthogonality of the bases functions, the equations \eqref{finite weak 1} and \eqref{finite weak 2} reduce to a coupled system of  nonlinear ordinary differential equations (ODEs) for the time-dependent coefficients $\alpha_k(t)$ and $\beta_k(t)$, for $k = 1,2,\cdots, n$,    are given by
\begin{align}
\frac{\mathrm{d} t \beta_k(t)}{\mathrm{d} t} &= \frac{1}{\|\boldsymbol{z}_k\|_{L^2}^2} \Bigg[ \br{\boldsymbol{f}(t)}{\boldsymbol{z}_k} - \mu_e \nr{\nabla \boldsymbol{z}_k} \beta_k(t) - \sum_{j=1}^{n} \Big( A_{kj}(\boldsymbol{\alpha}) + N_{kj}(\boldsymbol{\beta}) \Big) \beta_j(t) \Bigg], \label{eq:ode_velocity} \\ \frac{\mathrm{d} t \alpha_k(t)}{\mathrm{d} t} &= \frac{1}{\|w_k\|_{L^2}^2} \Bigg[ -\kappa \|w_k\|_{L^2}^2 \alpha_k(t) - D \|\boldsymbol{\nabla} w_k\|_{L^2}^2 \alpha_k(t) + Q_k(\boldsymbol{\alpha}) - \sum_{j=1}^{n}  C_{kj}(\boldsymbol{\beta}) \alpha_j(t) \Bigg]\label{eq:ode_conc}
\end{align}
where  $\boldsymbol{\alpha}:=(\alpha_1,\alpha_2,\cdots, \alpha_n)^{T}$, $\boldsymbol{\beta}:=(\beta_1,\beta_2,\cdots,\beta_n)^{T}$  are column vectors of coefficients, $Q_k(\boldsymbol{\alpha}) = \kappa \left( c_n^2, w_k \right)$ and the matrices $A_{kj},N_{kj},c_{kj}$  The matrices are given by 
\begin{align}A_{kj}(\boldsymbol{\alpha}) = \left( \frac{\widetilde{\mu}_{\ell}(c_n)}{K} \boldsymbol{z}_j, \boldsymbol{z}_k \right),   N_{kj}(\boldsymbol{\beta}) = \left( \beta |\boldsymbol{u}_n| \boldsymbol{z}_j, \boldsymbol{z}_k \right),  C_{kj}(\boldsymbol{\beta}) = \left( (\boldsymbol{u}_n \cdot \nabla) w_j, w_k \right). \end{align}
  The above system of ODEs is closed by specifying the following initial data $\alpha_k(0) = (c_0, w_k)$ and $\beta_k(0)= (\boldsymbol{u}_0, \boldsymbol{z}_k)$, for $k = 1,\cdots,n$. Since the nonlinear terms in the system (\ref{eq:ode_velocity})--(\ref{eq:ode_conc}) are locally Lipschitz continuous, the Picard–Lindelöf theorem guarantees the existence of a local solution. This completes the proof of the existence of the solution ${\bf u}_{n,\ell}$ and $c_{\ell, n}$ for each $n\in\mathbb{N}$. Next, we prove the necessary a priori bounds on solutions $\boldsymbol{u}_{\ell, n}$ and  $c_{\ell, n}$ for each fixed $\ell\in \mathbb{N}$. 
 
Since we are interested in proving the existence of solution for each fixed $\ell\in\mathbb{N}$ therefore for notational convenience in the subsequent analysis unless specified, we drop the $\ell$ subscript and write $c_n$ and $\boldsymbol{u}_n$ for $c_{n, \ell}$ and $\boldsymbol{u}_{\ell, n}$, respectively.

\begin{lem}\label{lemma1}
There exists a $T_{\max}\in (0,T)$ such that the  sequences $\{c_n\}_{n\in\mathbb{N}}$ and $\{{\bf u}_n\}_{n\in\mathbb{N}}$ are uniformly bounded in  $L^{\infty}(0,T_{{\max}};L^2(\Omega)) \cap L^{2}(0,T_{{\max}};H^1(\Omega))$ and $L^{\infty}(0,T_{{\max}};{\bf S})\cap L^{2}(0,T_{{\max}};\boldsymbol{V})$, respectively  i.e. there exist constants $M_1,M_2>0$ independent of $n$ such that \begin{align}\label{Bound on cn}
    \esssup_{t\in (0,T_{\max})}\nrm{c_{n}(t)}+\left(\int_{0}^{T_{\max}}\lVert c_n(t)\rVert_{H^1}^2\mathrm{d} t\right)^{1/2}\leq M_1, \ \mbox{for all}\ n\in\mathbb{N} 
\end{align}and \begin{align}\label{Bound on un} \esssup_{t\in (0,T_{\max})}\lVert{\bf u}_{n}(t)\rVert_{L^2}+\left(\int_{0}^{T_{\max}}\lVert {\bf u}_n(t)\rVert_{H^1}^2\mathrm{d} t\right)^{1/2}\leq M_2, \ \mbox{for all}\ n\in\mathbb{N}.\end{align}  
\end{lem}
\begin{proof}
We first prove the uniform boundedness of the sequence $\{c_n\}_{n\in\mathbb{N}}$. We multiply 
 equation \eqref{finite weak 2}, by $\alpha_k(t)$ and summing from  $k=1$ to $n$ and using equation \eqref{finite dim solution}, to  get  
\begin{align*}
 \innerproduct{\dfrac{\partial c_{n}(t)}{\partial t},c_{n}(t)} + \br{\boldsymbol{u}_n(t)\cdot \boldsymbol{\nabla} c_{n}(t)}{c_{n}(t)} + D \br{\boldsymbol{\nabla} c_{n}(t)} {\boldsymbol{\nabla} c_{n}(t)}  -  \br{\kappa\,c_n(t) (c_n(t) -1)}{c_{n}(t)} =0.    
\end{align*}
Note that the second term in the above equation vanishes because of the divergence free condition on velocity given by equation \eqref{model1} and the boundary condition on concentration given by equation \eqref{boundary conditions}. Therefore, the above equation reduces to  
\begin{align*}
    \frac{1}{2}\frac{\mathrm{d}}{\mathrm{d}t}\nr{c_n(t)} + D \nr{\boldsymbol{\nabla}c_n(t)} =  \int_{\Omega} \kappa\,(c_n(t))^3 -  \int_{\Omega}\kappa\,(c_n(t))^2.
\end{align*}
Applying Cauchy–Schwarz inequality to the first term on the right-hand side, and noting that $\kappa\, (c_n)^2 \geq 0$, we have
\begin{align*}
    \frac{1}{2}\frac{\mathrm{d}}{\mathrm{d}t}\nr{c_n(t)} + D \nr{\boldsymbol{\nabla}c_n(t)} \leq \kappa_2 \|c_n(t)\|_{L^4}^2 \nrm{c_n(t)}. 
\end{align*}
By using the  Gagliardo-Nirenberg inequality (see Theorem \ref{gagliardo}), we obtain
\begin{align*}
    \frac{1}{2}\frac{\mathrm{d}}{\mathrm{d}t}\nr{c_n(t)} + D \nr{\boldsymbol{\nabla}c_n(t)} \leq  \kappa_2 M   \|c_n(t) \|_{L^2}^{2(1-d/4)} \|c_n(t) \|_{H^1}^{d/2}\nrm{c_n(t)}, 
\end{align*}
where $d \in \{2,3 \}$. Now applying the Young's inequality (see theorem \ref{youngs inequaity}), we get 
\begin{align*}
    \frac{1}{2}\frac{\mathrm{d}}{\mathrm{d}t}\nr{c_n(t)} + D  \nr{\boldsymbol{\nabla}c_n(t)} \leq \epsilon \|c_n(t)\|_{H^1}^2 + \kappa_{2}^{4/(4-d)} \widetilde{M}_{\epsilon}^{4/(4-d)} \nrm{c_n(t)}^{2(6-d)/(4-d)}. 
\end{align*}
Now if  $\displaystyle M_{\epsilon} := \max\left\{\epsilon,\kappa_{2}^{4/(4-d)} \widetilde{M}_{\epsilon}^{4/(4-d)} \right\}$ then the above inequality becomes
\begin{align}\label{eq:concentration bound}
    \frac{1}{2}\frac{\mathrm{d}}{\mathrm{d}t}\nr{c_n(t)} + (D - \epsilon) \nr{\boldsymbol{\nabla}c_n(t)} \leq M_{\epsilon} \left( \nr{c_n(t)} +  \nrm{c_n(t)}^{2(6-d)/(4-d)} \right).
\end{align}
Now, by choosing $0<\epsilon < D/2$, and setting $ y_n(t) = \nr{c_n(t)}  $, the above inequality can be rewritten as 
\begin{align*}
    \frac{\mathrm{d} y_n(t) }{\mathrm{d} t}  \leq M_{\epsilon} \Big( y_n(t) + y_n^{(6-d)/(4-d)}(t)\Big), \text{ with } y_n(0) = \nr{c_n(0)}.
\end{align*}
By the comparison principle (see theorem \ref{Comparison principle}), $y_n(t)$ is bounded by the solution $z(t)$ of the corresponding ordinary differential equation
$$\frac{\mathrm{d}  z(t)}{\mathrm{d} t} = M_{\epsilon} ( z(t) + z^{(6-d)/(4-d)}(t) ), \text{ with } z(0) = \nr{c_{0}}.$$
Now using the Picard-Lindelöf theorem for ODEs, there exists a maximum time denoted by $T_{\max}\in (0, T)$ such that the above ODE has a unique solution $z(t)$ which exists on $[0, T_{\max})$, where $T_{\max}$ depends only on $M_\epsilon$ and $c_0$. Crucially, since $M_\epsilon$ and $c_0$ are independent of $n$, the local existence time $T_{\max}$ is also independent of $n$. Finally, using the comparison principle (see theorem \ref{Comparison principle}), we get 
\begin{align*}y_n(t) \leq z(t) \leq \mathcal{M}, \quad \text{for all } t \in [0, T_{\max}) \end{align*}
where the constant $\mathcal{M}$ is independent of $n$.
The bound $y_n(t) \le \mathcal{M}$ gives $\{c_n\}_{n\in\mathbb{N}}$ is uniformly bounded in $L^{\infty}(0,T_{\max};L^2(\Omega))$. Now  Integrating the equality \eqref{eq:concentration bound} in time from $0$ to $T_{\max}$, we have 
\begin{align*}
    2(D-\epsilon) \int_{0}^{T_{\max}} \nr{\boldsymbol{\nabla}c_n(t)} \leq  \nr{c_n(0)} - \nr{c_n(T_{\max})} + 2 M_{\epsilon}\int_{0}^{T_{\max}}  \left( \nr{c_n(t)} +  \nrm{c_n(t)}^{2(6-d)/(4-d)} \right).
\end{align*}
Now again by choosing $0<\epsilon < D/2$, and using $\{c_n\}_{n\in\mathbb{N}}$ is uniformly bounded in $L^{\infty}(0,T_{\max};L^2(\Omega))$, we conclude that $\{c_n\}_{n\in\mathbb{N}}$ is uniformly bounded in $L^{2}(0,T_{\max};H^1(\Omega))$. Thus uniform bound on $c_n$ given in inequality  \eqref{Bound on cn} is established.

Next we prove the uniform  boundedness of $\{{\bf u}_n\}_{n\in\mathbb{N}}$. We  multiply equation \eqref{finite weak 1}, by $\beta_k$, then summing form $k=1,2,...,n$, to get 
 \begin{align*}
  \innerproduct{\dfrac{\partial \boldsymbol{u}_{n}(t)}{\partial t}, \boldsymbol{u}_n(t)} + \left( \frac{\widetilde{\mu}_{\ell}(c_n(t))}{K}\boldsymbol{u}_n(t),\boldsymbol{u}_n(t)\right) + \left(\beta |\boldsymbol{u}_n(t)| \boldsymbol{u}_n(t),  \boldsymbol{u}_n(t)\right)+\mu_{e} \br{\boldsymbol{\nabla}\boldsymbol{u}_n(t)}{\boldsymbol{\nabla}\boldsymbol{u}_n(t)} - \br{\boldsymbol{f}(t)}{\boldsymbol{u}_n(t)}=0.     \end{align*}
Using the lower bounds for $\widetilde{\mu}_{\ell}$ and $\beta$ from \eqref{eq:trunc_viscosity} and \eqref{assumption:beta-k}, along with the  Cauchy–Schwarzinequality, we arrive at 
 \begin{align}\label{velocity bounds}
  \frac{1}{2}\frac{\mathrm{d}}{\mathrm{d} t} \nr{\boldsymbol{u}_n(t)} + \beta_1 \|\boldsymbol{u}_n(t)\|_{L^3}^{3} + \mu_e \nr{\boldsymbol{\nabla}\boldsymbol{u}_n(t)} \leq \nrm{\boldsymbol{f}(t)}\nrm{\boldsymbol{u}_n(t)}.
 \end{align}
Ignoring the non-negative terms from the left-hand side and applying  AM-GM inequality in the right side of the above inequality, we get  
 \begin{align*}
 \frac{1}{2} \frac{\mathrm{d}}{\mathrm{d} t} \nr{\boldsymbol{u}_n(t)}  \leq \frac{1}{2}\|\boldsymbol{f}(t)\|_{L^{2}}^2 + \frac{1}{2} \|\boldsymbol{u}_n(t)\|_{L^{2}}^2.
 \end{align*}
By Gronwall’s inequality, we obtain
\begin{align*}
    \nr{\boldsymbol{u}_n(t)}  \leq e^{T_{\max}}\Big( \nr{\boldsymbol{u}_n(0)} + \int_{0}^{T_{\max}} |\boldsymbol{f}(t)\|_{L^{2}}^2 \Big) .
\end{align*}
 Now using the fact that $\boldsymbol{f} \in L^{2}(0,T_{\max};L^2(\Omega))$, we get that $\{\boldsymbol{u}_n\}_{n\in\mathbb{N}}$ is uniformly bounded in $L^{\infty}(0,T_{\max};\boldsymbol{S})$. Finally integrating the inequality \ref{velocity bounds} from $0$ to $T_{\max}$ and using the uniform $L^\infty(0,T_{\max}; \boldsymbol{S})$ bound of $\{\boldsymbol{u}_n\}_{n\in\mathbb{N}}$, we find   that $\{\boldsymbol{u}_n\}_{n\in\mathbb{N}}$ is also uniformly bounded in $L^2(0,T_{\max};\boldsymbol{V})$. Hence, inequality \eqref{Bound on un} is derived. This completes the proof of lemma. 
\end{proof}
Next we prove the uniform bound on time-derivative of $\{c_n\}_{n\in\mathbb{n}}$ and ${\bf u}_{n\in\mathbb{N}}$ in some Sobolev norms. More precisely, we prove the following. 
\begin{lem}\label{lemma2}
The sequences $\big\{\frac{\partial c_n}{\partial t}\big\}_{n\in\mathbb{N}}$ and $\big\{\frac{\partial \boldsymbol{u}_n}{\partial t}\big\}_{n\in\mathbb{N}}$ are uniformly bounded in the space $L^{2}(0,T_{\max};(H^1(\Omega))^{\ast})$ and $L^{2}(0,T_{\max};\boldsymbol{V}^{\ast})$ respectively.
\end{lem}
\begin{proof} Using the projection theorem for Hilbert spaces,  any $\phi \in H^1(\Omega)$,  can be written as $\phi = \phi_n + \varphi_n$, where $\phi_n:= P_n \phi$ is the $H^1(\Omega)$-orthogonal projection of $\phi$ onto $W_n$ and $\varphi_n \in W_n^\perp$. Since $\partial_t c_n \in W_n$ by construction, it is orthogonal to $\varphi_n \in W_n^\perp$ therefore  $\innerproduct{\partial_t c_n, \varphi_n} = 0$. Now using $\innerproduct{\partial_t c_n, \varphi_n} = 0$ and triangle inequality in equation \eqref{finite weak 2}, we have 
\begin{align*}
 \left| \innerproduct{\frac{\partial c_n(t)}{\partial t}, \phi} \right| = \left| \innerproduct{\frac{\partial c_n(t)}{\partial t}, \phi_n} \right| \leq     |\br{c_n(t)\boldsymbol{u}_n(t)}{\boldsymbol{\nabla}\phi_n}| + D |\br{\boldsymbol{\nabla}c_n(t)}{\boldsymbol{\nabla}\phi_n}| + |\br{\kappa\, c_n(t)(c_n(t)-1)}{\phi_n} | .
\end{align*}
After using the  Cauchy–Schwarzinequality, the above inequality reduces to  
\begin{align*}
 \left| \innerproduct{\frac{\partial c_n(t)}{\partial t}, \phi} \right| \leq \|c_n(t)\|_{L^{4}}   \|\boldsymbol{u}_n(t) \|_{L^4} \nrm{\boldsymbol{\nabla}\phi_n} + D \nrm{\boldsymbol{\nabla}c_n(t)} \nrm{\boldsymbol{\nabla}\phi_n} + \kappa_2 \nrm{c_n(t)(c_n(t)-1)}\nrm{\phi_n}.
\end{align*}
Now using $\|\phi_n\|_{L^2} \leq \|\phi\|_{H^1}$ and $\|\boldsymbol{\nabla}\phi_n\|_{L^2} \leq \|\phi\|_{H^1}$, in the above inequality, we have
\begin{align*}
 \left| \innerproduct{\frac{\partial c_n(t)}{\partial t}, \phi} \right| \leq \|c_n(t)\|_{L^{4}}   \|\boldsymbol{u}_n(t) \|_{L^4} \|\phi\|_{H^1}  + D \nrm{\boldsymbol{\nabla}c_n(t)} \|\phi\|_{H^1} + \kappa_2 \nrm{c_n(t)(c_n(t)-1)}\|\phi\|_{H^1}.
\end{align*}
Taking the supremum over all $\phi \in H^1(\Omega)$ such that $\|\phi\|_{H^1} \leq 1$, we obtain
\begin{align*}
  \left\|\frac{\partial c_n(t)}{\partial t} \right\|_{(H^1(\Omega))^{\ast}} \leq \|c_n(t)\|_{L^{4}}   \|\boldsymbol{u}_n(t) \|_{L^4}  + D \nrm{\boldsymbol{\nabla}c_n(t)} + \kappa_2 \nrm{c_n(t)(c_n(t)-1)}.
\end{align*}
The terms on the right-hand side of the above inequality are uniformly bounded in $L^2(0, T_{\max})$. Therefore, we conclude that $\big\{\frac{\partial c_n}{\partial t}\big\}_{n\in\mathbb{N}}$ is uniformly bounded in $L^2(0,T_{\max};(H^1(\Omega))^{\ast})$. Utilizing the orthogonal decomposition $\boldsymbol{V} = \boldsymbol{V}_n \oplus \boldsymbol{V}_n^\perp$, any $\boldsymbol{v} \in \boldsymbol{V}$ can be uniquely expressed as $\boldsymbol{v} = \boldsymbol{v}_n + \boldsymbol{v}_n^\perp$. Since $\frac{\partial \boldsymbol{u}_n}{\partial t} \in \boldsymbol{V}_n$, it follows that $\left\langle \frac{\partial \boldsymbol{u}_n}{\partial t}, \boldsymbol{v}_n^\perp \right\rangle = 0$. Applying this orthogonality property and the triangle inequality to \eqref{finite weak 2}, we obtain
\begin{align*}
    \left| \innerproduct{\frac{\partial \boldsymbol{u}_n(t)}{\partial t}, \boldsymbol{v}} \right| = \left| \innerproduct{\frac{\partial \boldsymbol{u}_n(t)}{\partial t}, \boldsymbol{v}_n} \right| \leq \left|\left(\frac{\widetilde{\mu}(c_n)}{K}\boldsymbol{u}_n(t), \boldsymbol{v}_n  \right) \right|+ \left|\left(\beta |\boldsymbol{u}_n(t)|\boldsymbol{u}_n(t), \boldsymbol{v}_n  \right)\right| + \mu_e \left|\left(\boldsymbol{\nabla} \boldsymbol{u}_n(t), \boldsymbol{\nabla} \boldsymbol{v}_n  \right)\right| + \left|(\boldsymbol{f}(t), \boldsymbol{v}_n)\right|.
\end{align*}
By using the Cauchy-Schwarz inequality, and upper bounds for $\widetilde{\mu}(c_n)$ and $\beta$, we obtain 
\begin{align*}
  \left| \innerproduct{\frac{\partial \boldsymbol{u}_n(t)}{\partial t}, \boldsymbol{v}} \right| \leq   \frac{\mu_{\ell}}{K} \nrm{\boldsymbol{u}_n(t)} \nrm{\boldsymbol{v}_n} + \beta_{2} \|\boldsymbol{u}_n(t)\|_{L^4}^{2} \nrm{\boldsymbol{v}_n} + \mu_{e} \nrm{\boldsymbol{\nabla}\boldsymbol{u}_n(t)} \nrm{\boldsymbol{\nabla}\boldsymbol{v}_n} + \nrm{\boldsymbol{f}(t)} \nrm{\boldsymbol{v}_n}.
\end{align*}
After using $\nrm{\boldsymbol{v}_n} \leq \nrm{\boldsymbol{v}}$ and $\nrm{\boldsymbol{\nabla}\boldsymbol{v}_n} \leq \nrm{\boldsymbol{\nabla}\boldsymbol{v}}$, in the above inequality, we get 
\begin{align*}
  \left| \innerproduct{\frac{\partial \boldsymbol{u}_n(t)}{\partial t}, \boldsymbol{v}} \right| \leq   \frac{\mu_{\ell}}{K} \nrm{\boldsymbol{u}_n(t)} \nrm{\boldsymbol{v}} + \beta_{2} \|\boldsymbol{u}_n(t)\|_{L^4}^{2} \nrm{\boldsymbol{v}} + \mu_{e} \nrm{\boldsymbol{\nabla}\boldsymbol{u}_n(t)}\nrm{\boldsymbol{\nabla}\boldsymbol{v}} + \nrm{\boldsymbol{f}(t)} \nrm{\boldsymbol{v}}.
\end{align*}
Taking supremum over all $\boldsymbol{v} \in \boldsymbol{V}$ with $\|\boldsymbol{v}\|_{\boldsymbol{H}^1}\leq 1$ in the preceding inequality, we get
\begin{align*}
\left\| \frac{\partial \boldsymbol{u}_n(t)}{\partial t}\right\|_{\boldsymbol{V}^{\ast}}  \leq  \frac{\mu_{\ell}}{K} \nrm{\boldsymbol{u}_n(t)}  + \beta_{2} \|\boldsymbol{u}_n(t)\|_{L^4}^{2}  + \mu_{e} \nrm{\boldsymbol{\nabla}\boldsymbol{u}_n(t)} + \nrm{\boldsymbol{f}(t)} .
\end{align*}
Since all the terms on the right-hand side of the above inequality are uniformly bounded in $L^2(0, T_{\max})$, it follows that $\big\{\frac{\partial \boldsymbol{u}_n}{\partial t}\big\}_{n \in \mathbb{N}}$ is also uniformly bounded in $L^2(0, T_{\max}; \boldsymbol{V}^{*})$.
\end{proof}
\begin{lem}\label{lemma3}
 Under the assumptions of Theorem \ref{thm:existence}, there exists a weak solution $(c, \boldsymbol{u})$ to the truncated IBVP \eqref{model1}-\eqref{assumption:beta-k} on the interval $[0, T_{\max})$, where $T_{\max}$ is the time established in Lemma \ref{lemma1}.
\end{lem}
\begin{proof} From the uniform boundedness of $\{c_n\}_{n\in\mathbb{N}}$ and $\{{\bf u}_n\}_{n\in\mathbb{N}}$ established in Lemmas \ref{lemma1}-\ref{lemma2} and using the Banach-Alaoglu theorem, we have that there exist weakly convergent subsequences of $\{c_n\}_{n\in\mathbb{N}}$ and $\{{\bf u}_n\}_{n\in\mathbb{N}}$ still denoted by  $\{c_n\}_{n\in\mathbb{N}}$ and $\{{\bf u}_n\}_{n\in\mathbb{N}}$. Now if we denote by $c$ and $\boldsymbol{u}$ the subsequential limits of  $\{c_n\}_{n\in\mathbb{N}}$ and $\{{\bf u}_n\}_{n\in\mathbb{N}}$ respectively, then we have the following   
\begin{align*}
c_n &\rightharpoonup c \quad \  \text{weakly in } L^{2}\left(0,T_{\max};H^1(\Omega)\right), \\
\boldsymbol{u}_n &\rightharpoonup \boldsymbol{u} \quad\   \text{weakly in } L^{2}\left(0,T_{\max};\boldsymbol{V}\right), \\
\frac{\partial c_n}{\partial t} &\rightharpoonup \frac{\partial c}{\partial t} \quad  \text{weakly in } L^{2}\left(0,T_{\max};(H^{1}(\Omega))^*\right)\ \ \mbox{and}\\
\frac{\partial \boldsymbol{u}_n}{\partial t}  &\rightharpoonup \frac{\partial \boldsymbol{u}}{\partial t} \quad  \text{weakly in } L^{2}\left(0,T_{\max}; \boldsymbol{V}^{\ast}\right)
\end{align*}
as $n\rightarrow \infty$. Furthermore, the uniform bounds for $\displaystyle \{c_n\}_{n\in\mathbb{N}}$ in $\displaystyle L^2(0,T_{\max};H^1(\Omega))$ and $\displaystyle \left\{\frac{\partial c_n}{\partial t}\right\}_{n\in\mathbb{N}} $ in $\displaystyle L^2(0,T_{\max};(H^1(\Omega))^*)$ allow us to apply the Aubin-Lions Lemma stated in Theorem \ref{aubin-lions} to get the  following strong convergence for the aformentioned  subsequence $\{c_n\}_{n\in \mathbb{N}}$ 
\begin{align*}
c_n &\rightarrow c \quad \text{strongly in } L^2(0, T_{\max}; L^2(\Omega))\ \mbox{as $n\rightarrow \infty$}.
\end{align*}
A further consequence of this compactness result provides that the limit function $c$ satisfies  $c \in C([0, T_{\max}]; L^2(\Omega))$. Following a similar arguments as that of used for $\{c_n\}_{n\in\mathbb{N}}$, we conclude that  up to a subsequence $\boldsymbol{u}_n \rightarrow \boldsymbol{u}$ strongly in $L^2(0, T_{\max}; \boldsymbol{S})$ and  $\boldsymbol{u} \in C([0, T_{\max}]; \boldsymbol{S})$. Our next aim is to show that the limiting functions $(c,{\bf u})$ are weak solution to the truncated system of IBVP.  We need to prove that the limiting functions $(c,{\bf u})$ satisfy the weak formulations given by \eqref{weak1} and \eqref{weak2}  when $\mu$ is replaced by $\widetilde{\mu}_{\ell}$. To do so, we proceed as follows. Fix $m \leq n$ and  multiply  equation \eqref{finite weak 1} by $\eta \in C_{c}^{\infty}\left(0,T_{\max}\right) $ and integrating with respect to $t$  on $(0,T_{\max})$, we get that  
\begin{align}\label{passing the limit 1}
   & \int_{0}^{T_{\max}} \int_{\Omega} \frac{\partial \boldsymbol{u}_n(t)}{\partial t}\cdot \boldsymbol{v}\, \eta(t) \mathrm{d} x \mathrm{d} t+ \int_{0}^{T_{\max}} \int_{\Omega} \frac{\widetilde{\mu}_{\ell}(c_n(t))}{K}\boldsymbol{u}_n(t)\cdot \boldsymbol{v} \, \eta(t)\mathrm{d} x \mathrm{d} t +  \int_{0}^{T_{\max}} \int_{\Omega} \beta |\boldsymbol{u}_n(t)| \boldsymbol{u}_n(t)\cdot  \boldsymbol{v}\, \eta(t) \mathrm{d} x \mathrm{d} t \nonumber\\ & + \mu_e \int_{0}^{T_{\max}} \int_{\Omega} \boldsymbol{\nabla}\boldsymbol{u}_n(t):\boldsymbol{\nabla}\boldsymbol{v}\, \eta(t) \mathrm{d} x \mathrm{d} t -  \int_{0}^{T_{\max}} \int_{\Omega} \boldsymbol{f}(t) \cdot \boldsymbol{v}\, \eta(t)\mathrm{d} x \mathrm{d} t = 0,~ \forall \, \boldsymbol{v}\in \boldsymbol{V}_m.
\end{align}
Using Lemma \ref{lemma1} and the boundedness of $\widetilde{\mu}_{\ell}$, we have that  the sequence $\left\{\frac{\widetilde{\mu}_{\ell}(c_n)}{K}\boldsymbol{u}_n\right\}_{n\geq m}$ is uniformly bounded in $L^2(0,T_{\max}; \boldsymbol{L}^2(\Omega))$. Hence by the Banach-Alaoglu theorem, there exists a weakly convergent subsequence still denoted by $\left\{\frac{\widetilde{\mu}_{\ell}(c_n)}{K}\boldsymbol{u}_n\right\}_{n\geq m}$ and  $\boldsymbol{\chi} \in L^2(0, T_{\max}; \boldsymbol{L}^2(\Omega))$ such that
$$\frac{\widetilde{\mu}_{\ell}(c_n)}{K}\boldsymbol{u}_n \rightharpoonup \boldsymbol{\chi} \quad \text{weakly in } L^2(0, T_{\max}; \boldsymbol{L}^2(\Omega))\ \mbox{as $n\to \infty$}.$$ 

Now, since $\widetilde{\mu}_{\ell}$ is Lipschitz continuous and $c_n \to c$ strongly in $L^2(0, T_{\max}; L^2(\Omega))$ therefore we get   $\widetilde{\mu}_{\ell}(c_n) \to \widetilde{\mu}_{\ell}(c)$ strongly in $L^2(0, T_{\max}; L^2(\Omega))$. Thus we have $\frac{\widetilde{\mu}_{\ell}(c_n)}{K}\boldsymbol{u}_n$ converges to $\frac{\widetilde{\mu}_{\ell}(c)}{K}\boldsymbol{u}$ in $L^{1}(0, T_{\max}; \boldsymbol{L}^1(\Omega))$, hence using the uniqueness of limit, we get that $\boldsymbol{\chi} = \frac{\widetilde{\mu}_{\ell}(c)}{K}\boldsymbol{u}$.

\begin{align}\label{passing limit 2}
\lim_{n \to \infty}  \int_{0}^{T_{\max}} \int_{\Omega} \frac{\widetilde{\mu}_{\ell}(c_n(t))}{K}\boldsymbol{u}_n(t)\cdot \boldsymbol{v} \, \eta(t)\mathrm{d} x \mathrm{d} t =  \int_{0}^{T_{\max}} \int_{\Omega} \frac{\widetilde{\mu}_{\ell}(c(t))}{K}\boldsymbol{u}(t)\cdot \boldsymbol{v} \, \eta(t)\mathrm{d} x \mathrm{d} t,~ \mbox{for all}  \, \boldsymbol{v} \in \boldsymbol{V}_m .   
\end{align}
Next, we consider the third term on the left-hand side of equation \eqref{passing the limit 1}. Using Lemma \ref{lemma1}, we have that the sequence $\{\boldsymbol{u}_n\}$ is uniformly bounded in $L^2(0,T_{\max};{\bf V})$ therefore combining the Cauchy–Schwarz and Sobolev embedding theorem, we get that  $\{|\boldsymbol{u}_n|\boldsymbol{u}_n\}$  is uniformly bounded in $L^2(0,T_{\max};\boldsymbol{L}^2(\Omega))$. Hence an application of  the Banach-Alaoglu theorem, there exists a  subsequence still denoted by $\{|\boldsymbol{u}_n|\boldsymbol{u}_n\}$ and  $\boldsymbol{\psi} \in L^2(0,T_{\max};\boldsymbol{L}^2(\Omega))$ such that
$$|\boldsymbol{u}_n|\boldsymbol{u}_n \rightharpoonup \boldsymbol{\psi} \quad \text{weakly in } L^2(0,T_{\max};\boldsymbol{L}^2(\Omega))\ \mbox{as $n\rightarrow \infty$}.$$

Now using the strong convergence of $\boldsymbol{u}_n \to \boldsymbol{u}$  in $L^2(0,T_{\max}; \boldsymbol{L}^2(\Omega))$ and repeating the previous arguments, we get that $\psi=\lvert u\rvert u$. Using this, we get
\begin{align}\label{passing limit 3}\lim_{n \to \infty} \int_{0}^{T_{\max}} \int_{\Omega} \beta |\boldsymbol{u}_n(t)| \boldsymbol{u}_n(t)\cdot \boldsymbol{v}\, \eta(t) \mathrm{d} x \mathrm{d} t= \int_{0}^{T_{\max}} \int_{\Omega} \beta |\boldsymbol{u}(t)| \boldsymbol{u}(t)\cdot \boldsymbol{v}\, \eta(t)\mathrm{d} x \mathrm{d} t, ~ \mbox{for all}\ \ \boldsymbol{v} \in \boldsymbol{V_m}.
\end{align}
Now using the fact that $m$ is  arbitrary and the union $\cup_{m=1}^{\infty} \boldsymbol{V}_m$ is dense in $\boldsymbol{V}$,  the convergence results \eqref{passing limit 2} and \eqref{passing limit 3} hold for all $\boldsymbol{v} \in \boldsymbol{V}$, hence taking limit $n\rightarrow \infty$ in 
 \eqref{passing the limit 1},  we  get
\begin{align}\label{Limit passing for velocity field}
 \int_{0}^{T_{\max}} \left[\innerproduct{\frac{\partial \boldsymbol{u}(t)}{\partial t}, \boldsymbol{v}} + \left( \frac{\widetilde{\mu}_{\ell}(c(t))}{K}\boldsymbol{u}(t),\boldsymbol{v}\right) + \left(\beta |\boldsymbol{u}(t)| \boldsymbol{u}(t),  \boldsymbol{v}\right)+\mu_{e} \br{\boldsymbol{\nabla}\boldsymbol{u}(t)}{\boldsymbol{\nabla}\boldsymbol{v}} - \br{\boldsymbol{f}(t)}{\boldsymbol{v}}\right]\eta(t)  \mathrm{d} t = 0,~ \mbox{for all}\ \, \boldsymbol{v}\in \boldsymbol{V}.
\end{align}
Now the above equality holds for all $\eta \in C_{c}^{\infty}(0,T_{\max})$, therefore we have 
\begin{align*}
 \innerproduct{\frac{\partial \boldsymbol{u}(t)}{\partial t}, \boldsymbol{v}} + \left( \frac{\widetilde{\mu}_{\ell}(c(t))}{K}\boldsymbol{u}(t),\boldsymbol{v}\right) + \left(\beta |\boldsymbol{u}(t)| \boldsymbol{u}(t),  \boldsymbol{v}\right)+\mu_{e} \br{\boldsymbol{\nabla}\boldsymbol{u}(t)}{\boldsymbol{\nabla}\boldsymbol{v}} - \br{\boldsymbol{f}(t)}{\boldsymbol{v}} = 0,~ a.e. \text{ on } (0,T_{\max}), ~ \mbox{for all}\ \, \boldsymbol{v}\in \boldsymbol{V}.
\end{align*}
Next, we pass to the limit in equation \eqref{finite weak 2}. We first observe that the sequence $\{c_n \boldsymbol{u}_n\}_{n \in \mathbb{N}}$ is uniformly bounded in $L^2(0, T_{\max};\boldsymbol{L}^2(\Omega))$. Consequently, by the Banach-Alaoglu theorem, there exists a subsequence, still denoted by $\{c_n \boldsymbol{u}_n\}_{n\in\mathbb{N}}$, and a function $\boldsymbol{\chi}$ such that $c_n \boldsymbol{u}_n \rightharpoonup \boldsymbol{\chi}$ weakly in $L^2(0, T_{\max};\boldsymbol{L}^2(\Omega))$. Given that $c_n \to c$ and $\boldsymbol{u}_n \to \boldsymbol{u}$ strongly in $L^2(0,T_{\max};L^2(\Omega))$ and $L^2(0,T_{\max};\boldsymbol{L}^2(\Omega))$ respectively as $n \to \infty$, it follows that $c_n \boldsymbol{u}_n \to c\boldsymbol{u}$ strongly in $L^1(0,T_{\max};\boldsymbol{L}^1(\Omega))$. By the uniqueness of the weak limit, we identify $\boldsymbol{\chi} = c\boldsymbol{u}$, and thus $c_n \boldsymbol{u}_n \rightharpoonup c\boldsymbol{u}$ weakly in $L^2(0, T_{\max};\boldsymbol{L}^2(\Omega))$. Similarly, regarding the nonlinear reaction term, since $\{c^2_n\}_{n\in\mathbb{N}}$ is uniformly bounded in $L^2(0,T_{\max};L^2(\Omega))$, therefore we obtain up to a subsequence  $c^2_n \rightharpoonup c^2$ weakly in $L^2(0,T_{\max};L^2(\Omega))$. As the remaining terms in \eqref{finite weak 2} are linear, we apply arguments analogous to those used for equation \eqref{Limit passing for velocity field} to obtain
\begin{align*}
  \innerproduct{\frac{\partial c(t)} {\partial t},\phi} - \br{c(t)\boldsymbol{u}(t)}{\boldsymbol{\nabla}\phi} + D \br{\boldsymbol{\nabla}c(t)}{\boldsymbol{\nabla}\phi} = \br{\kappa\, c(t)(c(t)-1)}{\phi},~ a.e. \text{ on } (0,T_{\max}), ~ \mbox{for all}\ \, \boldsymbol{\phi}\in H^1(\Omega).
\end{align*}
This establishes the existence of a local solution $(c_{\ell}, \boldsymbol{u}_{\ell})$ for the truncated IBVP system \eqref{model1}-\eqref{assumption:beta-k}, thereby completing the proof of Lemma \ref{lemma3}.
\end{proof}
We prove the existence of global solution to the system of truncated IBVP \eqref{model1}-\eqref{assumption:beta-k}, by proving the maximum-minimum principles for concentration $c_{\ell}$.  In the next lemma, we prove the maximum-minimum principles for concentration $c_{\ell}$. 


\begin{lem}\label{max lemma}
Let $ (c_{\ell}, \boldsymbol{u}_{\ell}) $ be a weak solution to the truncated system of IBVP   \eqref{model1}–\eqref{assumption:beta-k}. Now if $$0 \leq c_0(\boldsymbol{x})   < 1 \quad \text{for almost every } \boldsymbol{x} \in \Omega
$$
then the concentration $c_{\ell}$ satisfies the following 
$$
0 \leq c_{\ell}(t, \boldsymbol{x}) < 1, \quad \text{for $t\in (0,T_{\max})$ and almost every } \ x \in  \Omega.
$$
\end{lem}
\begin{proof} To establish the non-negativity of the concentration, we define $ \widetilde{c}_{\ell} := \min\{0,c_{\ell}\}$. Subsequently, by substituting $\phi =\widetilde{c}_{\ell}$ into equation \eqref{weak2}, we obtain
\begin{align*}
    \innerproduct{\frac{\partial c_{\ell}(t)} {\partial t},\widetilde{c}_{\ell}(t)} + \br{\boldsymbol{u}(t)\cdot \boldsymbol{\nabla}c_{\ell}(t)}{\widetilde{c}_{\ell}(t)} + D \br{\boldsymbol{\nabla}c_{\ell}(t)}{\boldsymbol{\nabla}\widetilde{c}_{\ell}(t)} = \br{\kappa\, c_{\ell}(t)(c_{\ell}(t)-1)}{\widetilde{c}_{\ell}(t)}.
\end{align*}
Using the the definition of $\tilde{c_{\ell}}$, the above equation reduce to
\begin{align*}
\frac{1}{2}\frac{\mathrm{d}}{\mathrm{d} t}\nr{\widetilde{c}_{\ell}(t)}  + \frac{1}{2}\br{\boldsymbol{u}}{\boldsymbol{\nabla}(\widetilde{c}_{\ell}(t))^2} + D \nr{\boldsymbol{\nabla}\widetilde{c}_{\ell}(t)}  = \int_{\{c_{\ell} \leq 0\}} \kappa \, c_{\ell}^{2}(t)(c_{\ell}(t)-1) \mathrm{d} x .
\end{align*}
Applying the divergence theorem, the incompressibility condition $\nabla \cdot \boldsymbol{u}_{\ell} = 0$ and the boundary conditions \eqref{boundary conditions}, we find that the convective term vanishes. Furthermore, noting that the right-hand side is non-positive and dropping the non-negative term on the left-hand side, we obtain
$$\frac{1}{2}\frac{\mathrm{d}}{\mathrm{d} t}\|\widetilde{c}_{\ell}(t)\|_{L^2}^2 \leq 0.$$
Integrating with respect to $t$ from $0$ to $\tau \in (0,T_{\max})$ yields 
\begin{align*}
\nr{\widetilde{c}_{\ell}(\tau)}  \leq \nr{\widetilde{c}_{\ell}(0)}, \quad \text{ for all } \tau \in (0, T_{\max}).
\end{align*}
Now  $\widetilde{c}_{\ell}(0) = 0$ gives that $\widetilde{c}_{\ell}(t,x)=0$ for $t\in (0,T_{\max})$ and almost every  $x\in\Omega$, hence using the definition of $\widetilde{c}_{\ell}$,  we get $c_{\ell}(t,\boldsymbol{x}) \geq 0$ for $t\in (0,T_{\max})$ and almost every  $x\in\Omega$.
Next to establish the bound $c_{\ell}(t, \boldsymbol{x}) < 1$ for $t\in (0,T_{\max})$ and  almost every $x\in \Omega$, we proceed by contradiction. Suppose if possible that this assertion is false, then there exists a time $t_0 \in (0, T_{\max})$ and a set of positive measure $\Omega_0 \subset \Omega$ such that
$$c_{\ell}(t_0, \boldsymbol{x}) \geq 1 \quad \text{for a.e. } \boldsymbol{x} \in \Omega_0,$$
while for all $t \in [0, t_0)$,$$c_{\ell}(t, \boldsymbol{x}) < 1 \quad \text{for a.e. } \boldsymbol{x} \in \Omega.$$
Let \( \psi \in C_c^\infty(\Omega_0) \) be a smooth cut-off function such that \( 0 \leq \psi \leq 1 \) in \( \Omega_0 \). Using \(\phi = c_{\ell}(t)\psi^2\) in equation \eqref{weak2}, we have 
\begin{align}\label{max:eq1}
\innerproduct{\frac{\partial c_{\ell}(t)} {\partial t},c_{\ell}(t)\psi^2} + \br{\boldsymbol{u}(t)\cdot\boldsymbol{\nabla}c_{\ell}(t)}{c_{\ell}(t)\psi^2} + D \br{\boldsymbol{\nabla}c_{\ell}(t)}{\boldsymbol{\nabla}(c_{\ell}(t)\psi^2)} = \br{\kappa\, c_{\ell}(t)(c_{\ell}(t)-1)}{c_{\ell}(t)\psi^2} .
\end{align} 
We now analyze each term of the above equation separately. Since $\psi$ is independent of time, the first term becomes
\begin{align}\label{max:eq2}
   \innerproduct{\frac{\partial c_{\ell}(t)} {\partial t},c_{\ell}(t)\psi^2} = \frac{1}{2} \frac{\mathrm{d}  }{\mathrm{d} t}\nr{c_{\ell}(t)\psi}.
\end{align}
We rewrite the convection term in the following way
\begin{align*}
    \br{\boldsymbol{u}(t)\cdot\boldsymbol{\nabla}c_{\ell}(t)}{c_{\ell}(t)\psi^2} = \frac{1}{2} \int_{\Omega} \boldsymbol{u}(t)\cdot \boldsymbol{\nabla}(c_{\ell}(t)\psi^2) \mathrm{d} x  - \frac{1}{2} \int_{\Omega} c_{\ell}^2(t) \boldsymbol{u}(t)\cdot \boldsymbol{\nabla}\psi^2 \mathrm{d} x .
\end{align*}
Applying integration by parts and utilizing equations \eqref{model1} and \eqref{boundary conditions}, the first term vanishes. Consequently, the equation reduces to
\begin{align}\label{max:eq3}
    \br{\boldsymbol{u}(t)\cdot\boldsymbol{\nabla}c_{\ell}(t)}{c_{\ell}(t)\psi^2} =  - \frac{1}{2} \int_{\Omega} c_{\ell}^2(t) \boldsymbol{u}(t)\cdot \boldsymbol{\nabla}\psi^2 \mathrm{d} x .
\end{align}
For the diffusion term, we expand the gradient of the test function as $\boldsymbol{\nabla}(c_{\ell}(t)\psi^2) = \psi^2\boldsymbol{\nabla}c_{\ell}(t) + 2c(t)\psi\boldsymbol{\nabla}\psi$, which yields
\begin{align}\label{max:eq4}
\br{\boldsymbol{\nabla}c_{\ell}(t)}{\boldsymbol{\nabla}(c_{\ell}(t)\psi^2)} = \int_{\Omega} \left( \psi^2 |\boldsymbol{\nabla}c_{\ell}(t)|^2 + 2c(t)\psi \boldsymbol{\nabla}c_{\ell}(t) \cdot \boldsymbol{\nabla}\psi \right) \mathrm{d} x.
\end{align}
To further simplify this term, taking the square on both sides of  $\boldsymbol{\nabla}(c_{\ell}(t)\psi) = \psi\boldsymbol{\nabla}c_{\ell}(t) + c_{\ell}(t)\boldsymbol{\nabla}\psi$, we obtain
\begin{align*}
  2 c_{\ell}(t) \psi \boldsymbol{\nabla} \psi \cdot \boldsymbol{\nabla} c_{\ell}(t)  =   |\boldsymbol{\nabla}(c_{\ell}(t) \psi )|^2 -  c_{\ell}^2(t) |\boldsymbol{\nabla} \psi|^2 - \psi^2 |\boldsymbol{\nabla}c_{\ell}(t)|^2. 
\end{align*}
Using this  in equation \eqref{max:eq4}, we get 
\begin{align}\label{max:eq5}
  \br{\boldsymbol{\nabla}c_{\ell}(t)}{\boldsymbol{\nabla}(c_{\ell}(t)\psi^2)} = \int_{\Omega} \left( |\boldsymbol{\nabla}(c_{\ell} \psi)|^2 -  c_{\ell}^2 |\boldsymbol{\nabla} \psi|^2 \right)\mathrm{d} x.
\end{align}
The reaction term can be written as
\begin{align}\label{max:eq6}
  \br{\kappa\, c_{\ell}(t)(c_{\ell}(t)-1)}{c_{\ell}(t)\psi^2}  =  \int_{\Omega} \kappa \, \psi^2 c_{\ell}^2(t) \left(c_{\ell}(t)-1 \right)\mathrm{d} x.
\end{align}
Combining   equations \eqref{max:eq2}-\eqref{max:eq6} in equation \eqref{max:eq1}, we get 
\begin{align*}
\frac{1}{2} \frac{\mathrm{d}}{\mathrm{d} t}\nr{c_{\ell}(t)\psi} + \int_{\Omega}|\boldsymbol{\nabla}(c_{\ell}(t) \psi)|^2 \mathrm{d} x  =    \frac{1}{2} \int_{\Omega} c_{\ell}^2(t) \boldsymbol{u}(t)\cdot \boldsymbol{\nabla}\psi^2 \mathrm{d} x  + \int_{\Omega} c_{\ell}^2 (t)|\boldsymbol{\nabla} \psi|^2 \mathrm{d} x +  \int_{\Omega} \kappa\, \psi^2 c_{\ell}^2(t) \left(c_{\ell}(t)-1 \right) \mathrm{d} x.
\end{align*}
For $\epsilon>0$ to be specified later, we integrate  the above equality from $t_{0} - \epsilon$ to $t_0$ to obtain
\begin{align*}
 &\nr{c_{\ell}(t_0)\psi} + 2 \int_{t_{0}-\epsilon}^{t_{0}} \int_{\Omega}|\boldsymbol{\nabla}(c_{\ell}(t) \psi)|^2 \mathrm{d} x \mathrm{d} t \\&= \int_{t_{0}-\epsilon}^{t_{0}} \int_{\Omega} c_{\ell}^2(t) \boldsymbol{u}(t)\cdot \boldsymbol{\nabla}\psi^2 \mathrm{d} x \mathrm{d} t + 2 \int_{t_{0}-\epsilon}^{t_{0}} \int_{\Omega} c_{\ell}^2(t) |\boldsymbol{\nabla} \psi|^2 \mathrm{d} x \mathrm{d} t + 2  \int_{t_{0}-\epsilon}^{t_{0}} \int_{\Omega} \kappa\, \psi^2 c_{\ell}^2(t) \left(c_{\ell}(t)-1 \right)\mathrm{d} x \mathrm{d} t + \nr{c_{\ell}(t_{0} - \epsilon)\psi}.
\end{align*}
We now utilize the fact that $c_{\ell}(t,x) < 1$ for all $t\in (t_0 - \epsilon, t_{0})$.
\begin{align*}
   \nr{c_{\ell}(t_0)\psi}  + 2 \int_{t_{0}-\epsilon}^{t_{0}} \int_{\Omega}|\boldsymbol{\nabla}(c_{\ell}(t) \psi)|^2 \mathrm{d} x \mathrm{d} t < \left(\|\boldsymbol{u}\|_{L^{\infty}(t_{0}-\epsilon, t_{0};\boldsymbol{S})} \nrm{\nabla \psi^2} + 2 \nr{\boldsymbol{\nabla}\psi}\right) \epsilon +  \nr{c_{\ell}(t_{0} - \epsilon)\psi}.
\end{align*}
Now using the continuity of $c_{\ell}$ in $t$ variable and the fact that $c_{\ell}(t_0,x)\geq 1$ for almost every $x\in \Omega_0$, we can choose $\epsilon>0$ small enough such that \[\nr{\psi}\leq \nr{c_{\ell}(t_0)\psi} \leq    \nr{c_{\ell}(t_{0} - \epsilon)\psi}<\nr{\psi} \]
where  in the last step the $\lVert\cdot\rVert_{L^2}$ norm is on $\Omega$ because of the assumption on $\psi$. 
This yields a contradiction, hence $c_{\ell}(t,x)<1$ for $t\in (0,T_{\max})$ and almost every $x\in\Omega$. 
\end{proof}
\begin{remark}
This lemma proves that if we assume $M \in (0,1)$ and $0 \leq c_0 \leq M $, then $0 \leq c_{\ell}(t,x) \leq M$ for almost all $(t,x) \in (0,T) \times \Omega$.    
\end{remark} 
Next, with the help of the maximum-minimum principles proved in Lemma \ref{max lemma}, we will show that the local weak solution established in Lemma \ref{lemma3} can be extended from $[0, T_{\max})$ to  $(0, T)$ for any $T>0$. This will be shown in the following lemma. 

\begin{lem}\label{Global existence for truncated}
    The local weak solution $(c_{\ell}, \boldsymbol{u}_{\ell})$ established in Lemma \ref{lemma3} on the interval $(0, T_{\max})$ can be extended to the full interval $(0, T)$ for any $T>0$. 
\end{lem}
\begin{proof}
 To extend the local solution $(c_{\ell},\boldsymbol{u}_{\ell})$ to the entire interval $(0, T)$ for any $T>0$, we employ the continuation argument. Let if possible assume  that the maximal time of existence is $0<T^{\ast} < T,$ then we have that
$$\lim_{t \to T^{\ast-}} \|c_{\ell}(t)\|_{L^2(\Omega)} = \infty.$$
Now utilizing the bound $0 \leq c_{\ell} < 1$ and the estimates established in Lemmas \ref{lemma1} and \ref{lemma2}, we assert that $c_{\ell} \in L^{2}(0,T^{\ast};H^1(\Omega))$ and $\partial_t c_{\ell} \in L^{2}(0,T^{\ast};(H^1(\Omega))^{\ast})$. By the Aubin-Lions lemma, this regularity guarantees that $c_{\ell} \in C([0,T^{\ast}]; L^2(\Omega))$. Furthermore, repeating the previous argument, we obtain that 
\begin{align*}
\|c_{\ell}(t)\|_{L^2(\Omega)} \leq \|c_0\|_{L^2(\Omega)} e^{\kappa_2 T^{\ast}} \quad \text{for all } t \in [0, T^{\ast}).
\end{align*}
This uniform estimate contradicts the blow-up hypothesis. Consequently, we must have $T^* = T$, and the local solution can be extended to the entire interval $(0, T)$ for any $T>0$. Thus, for any fixed truncation parameter $\ell > 0$, we have obtained a global solution $(c_{\ell}, \boldsymbol{u}_{\ell})$ to the truncated IBVP \eqref{model1}-\eqref{assumption:beta-k}.
\end{proof}
 To complete the proof of Theorem \ref{thm:existence}, it remains to show that the limit pair $(c, \boldsymbol{u})$, obtained as $\ell \to \infty$, is a weak solution to the original IBVP \eqref{model1}-\eqref{assumption:beta-k}.

\subsection{Proof of Theorem \ref{thm:existence}} Using Lemma  \ref{max lemma}, we observe that for each $\ell\in\mathbb{N}$,   $0 \leq c_\ell(t,x) < 1$ for  $t\in (0,T)$ and almost every $x\in \Omega$. Hence,  the sequence $\{c_{\ell}\}_{\ell\in\mathbb{N}}$  is uniformly bounded and after utilizing this along with repeating the arguments used in  Lemmas \ref{lemma1} and \ref{lemma2}, we get that $\{c_{\ell}\}_{\ell \in \mathbb{N}}$ and $\displaystyle \left\{\frac{\partial_t c_{\ell}}{\partial t}\right\}_{\ell \in \mathbb{N}}$ are uniformly bounded in $L^2(0,T;H^1(\Omega))$ and $L^2(0,T; (H^1(\Omega))^{\ast})$, respectively. After adapting the similar arguments for velocity field, we get that   $\{\boldsymbol{u}_{\ell}\}_{\ell\in\mathbb{N}}$ and $\displaystyle \left\{\frac{\partial_t \boldsymbol{u}_{\ell}}{\partial t}\right\}_{\ell \in \mathbb{N}}$ are uniformly bounded in $L^2(0,T;\boldsymbol{V})$ and $L^2(0,T;\boldsymbol{V}^{\ast})$ respectively. These uniform bounds allow us to apply the Aubin-Lions lemma to extract subsequences still denoted by $\{c_\ell\}_{\ell\in\mathbb{N}}$ and $\{\boldsymbol{u}_\ell\}_{\ell\in\mathbb{N}}$ such that $c_\ell \to c$ and $\boldsymbol{u}_\ell \to \boldsymbol{u}$ strongly in $L^2(0, T; L^2(\Omega))$ and $L^2(0, T; \boldsymbol{L}^2(\Omega))$ as $\ell \to \infty$ respectively. 
Next, using the triangle inequality, we get 
\begin{align}\label{Lipschitzness viscosity}
\|\widetilde{\mu}_{\ell}(c_{\ell}) - \mu(c)\|_{L^2} \leq \|\widetilde{\mu}_{\ell}(c_{\ell}) - \mu(c_{\ell})\|_{L^2} + \|\mu(c_{\ell}) - \mu(c)\|_{L^2}
\end{align}
we show that  $\widetilde{\mu}_{\ell}(c_{\ell}) \to \mu(c)$ strongly in $L^2(0, T; L^2(\Omega))$ as $\ell \to \infty$. Now using Lemma \ref{max lemma} along with the definition of $\widetilde{\mu}$, we get that $\widetilde{\mu}_{\ell}(c_{\ell}) = \mu(c_{\ell})$ for all $\ell \geq 1$, which amounts in vanishing of first term on the right-hand side of inequality \eqref{Lipschitzness viscosity}. Finally, using the local Lipschitzness of $\mu$ along with the uniform bound on $\{c_{\ell}\}_{\ell\in\mathbb{N}}$ and its strong convergence to $c$ in $L^2(0, T; L^2(\Omega))$, we get  that $\mu(c_{\ell}) \to \mu(c)$ in $L^2(0,T; L^2(\Omega))$ as $\ell \to \infty$. Combining these observations, we conclude from equation \eqref{Lipschitzness viscosity} that $\widetilde{\mu}_{\ell}(c_{\ell}) \to \mu(c)$ strongly in $L^2(0, T; L^2(\Omega))$ as $\ell \to \infty$ which along with the arguments employed in Lemma \ref{lemma3}, gives that the limit $(c, \boldsymbol{u})$ of $\left(\{c_\ell\}_{\ell\in\mathbb{N}},\{{\bf u}_{\ell}\}_{\ell\in\mathbb{N}}\right)$  constitutes a weak solution to the  IBVP \eqref{model1}-\eqref{assumption:beta-k}. This concludes the proof of Theorem \ref{thm:existence}.

\subsection{Continuous Dependence and Uniqueness of Weak Solutions}\label{sec: continous dependence and uniqueness}
\begin{thm}[Continuous dependence on initial data]\label{Thm:continous dependence}
Let $(c_1,\boldsymbol{u}_1)$ and $(c_2,\boldsymbol{u}_2)$ be two weak solutions of IBVP \eqref{model1}-\eqref{assumption:beta-k} with the initial data
$c_{1}(0,\boldsymbol{x})$ and $c_{2}(0,\boldsymbol{x})$, respectively. Then for all $t \in [0,T]$, we have
    \begin{align*}
       \nr{c_1(t) - c_2(t)} + \nr{\boldsymbol{u}_1(t) - \boldsymbol{u}_2(t)} \leq \left( \nr{c_1(0) - c_2(0)} + \nr{\boldsymbol{u}_1(0) - \boldsymbol{u}_2(0)} \right) \exp{\left(2 \int_{0}^{T} \Phi(\tau) d \tau\right)},
\end{align*}
for some nonnegative  function  $\Phi\in L^{1}(0,T)$ depending on $c_1,c_2$, ${\bf u}_{1},{\bf u}_{2}$, $\mu$ and $\kappa$. 
\end{thm}

\begin{proof} We start with denoting $c:= c_1 - c_2$ and $\boldsymbol{u}:= \boldsymbol{u}_1 - \boldsymbol{u}_2$ and subtracting the two sets of the weak formulations of  equations corresponding to $c_i$ and ${\bf u}_{i}$ for $i=1,2$, to arrive at the following  equations for $c$   and ${\bf u}$
\begin{align}
\begin{aligned}\label{uniqueness:eq1}
 &\innerproduct{\frac{\partial c(t)}{\partial t},\phi} + \br{\boldsymbol{u_1}(t)\cdot \boldsymbol{\nabla}c_1(t) - \boldsymbol{u_2}(t)\cdot \boldsymbol{\nabla}c_2(t)}{\phi} +D \br{\boldsymbol{\nabla}c(t)}{\boldsymbol{\nabla}\phi} \\
 &\qquad = \br{\kappa\, c_1(t)(c_1(t)-1) - \kappa \,c_2(t)(c_2(t)-1)}{\phi}, \ \mbox{for all $t \in (0,T)$ and all $\phi \in H^1(\Omega)$}
 \end{aligned}
\end{align}
and 
\begin{align}\label{uniqueness:eq3}
  &\innerproduct{\frac{\partial\boldsymbol{u}(t)}{\partial t},\boldsymbol{v}} +  \br{\mu(c_1(t))\boldsymbol{u}_1(t)-\mu(c_2(t))\boldsymbol{u}_2(t)}{\boldsymbol{v}} \nonumber  +  \br{\beta|\boldsymbol{u}_1(t)|\boldsymbol{u}_1(t) -\beta|\boldsymbol{u}_2(t)|\boldsymbol{u}_2(t)}{\boldsymbol{v}} \\ & \qquad \qquad\qquad \qquad \qquad  + \mu_{e} \br{\boldsymbol{\nabla}\boldsymbol{u}(t)}{\boldsymbol{\nabla}\boldsymbol{v}} = 0,  \ \mbox{for all}\  t \in (0,T) \ \mbox{and all} \ \boldsymbol{v} \in H^1(\Omega). 
\end{align}
respectively. Substituting  $\phi = c$ in equation \eqref{uniqueness:eq1} leads to
\begin{align*}
 \frac{1}{2}\frac{\mathrm{d}}{\mathrm{d} t}\nr{c(t)} + \br{\boldsymbol{u_1}(t)\cdot \boldsymbol{\nabla}c_1(t) - \boldsymbol{u_2}(t)\cdot \boldsymbol{\nabla}c_2(t)}{c(t)} +D \nr{\boldsymbol{\nabla}c(t)}  = \br{\kappa c_1(t)(c_1(t)-1)-\kappa c_2(t)(c_2(t)-1)}{c(t)}.
\end{align*}
Rearranging the terms of the above equation, we arrive at 
\begin{align*}
  \frac{1}{2}\frac{\mathrm{d}}{\mathrm{d} t}\nr{c(t)} +   \br{\boldsymbol{u_1}(t)\cdot \boldsymbol{\nabla}c(t)}{c(t)} + \br{\boldsymbol{u}(t)\cdot \boldsymbol{\nabla}c_2(t)}{c(t)} +D \nr{\boldsymbol{\nabla}c(t)} =  \br{\kappa\,(c_1(t) + c_2(t) -1)c(t)}{c(t)}.
\end{align*}
The second term on the left-hand side vanishes upon applying the divergence theorem, the boundary conditions \eqref{boundary conditions}, and the incompressibility constraint \eqref{model1}. Subsequently, applying Cauchy–Schwarz and Young's inequalities, we obtain
\begin{align*}
\frac{1}{2}\frac{\mathrm{d}}{\mathrm{d} t}\nr{c(t)} + D \nr{\boldsymbol{\nabla}c(t)}  \leq \epsilon \nr{\boldsymbol{\nabla}c(t)}  + M_{\epsilon} \norm{c_2(t)}^2  \nr{\boldsymbol{u}(t)}  + \kappa_2 \norm{c_1(t) + c_2(t) -1}\nr{c(t)}.
\end{align*}
Now, utilizing the bounds $0 \leq c_1, c_2 \leq 1$, the above inequality reduces to
\begin{align}\label{uniqueness:eq2}
\frac{1}{2}\frac{\mathrm{d}}{\mathrm{d} t}\nr{c(t)} + (D-\epsilon) \nr{\boldsymbol{\nabla}c(t)}  \leq  M_{\epsilon}   \nr{\boldsymbol{u}(t)}  + \kappa_2 \nr{c(t)}.
\end{align}
Now substituting  $\boldsymbol{v} = \boldsymbol{u}$ in equation \eqref{uniqueness:eq2} and rearranging the terms, we obtain
\begin{align}\label{velocity uniqueness}
  &\frac{1}{2}\frac{\mathrm{d}}{\mathrm{d} t} \nr{\boldsymbol{u}(t)} + \br{\mu(c_1(t))\boldsymbol{u}(t)}{\boldsymbol{u}(t)}  + \mu_{e}\nr{\boldsymbol{\nabla}\boldsymbol{u}(t)}\nonumber \\ & \qquad +    \br{\beta(|\boldsymbol{u}_1(t)|\boldsymbol{u}_1(t)-|\boldsymbol{u}_2(t)|\boldsymbol{u}_2(t))}{\boldsymbol{u}(t)} + \br{(\mu(c_1(t))-\mu(c_2(t)))\boldsymbol{u}_2(t)}{\boldsymbol{u}(t)}=0.
\end{align}
Utilizing the monotonicity of the Forchheimer term and the non-negative preserving property of $\mu$, we get 
\begin{align*}
 \mu(c_1) \ge 0 \quad \text{and} \quad \big(|\boldsymbol{u}_1|\boldsymbol{u}_1 - |\boldsymbol{u}_2|\boldsymbol{u}_2\big) \cdot (\boldsymbol{u}_1 - \boldsymbol{u}_2) \geq 0, \quad \text{a.e. in } (0,T) \times \Omega.   
\end{align*}
 Dropping the non-negative terms from \eqref{velocity uniqueness} and applying the Lipschitz continuity of $\mu$ along with Cauchy–Schwarz inequality to the last term, we obtain
\begin{align*}
 \frac{1}{2}\frac{\mathrm{d}}{\mathrm{d} t} \nr{\boldsymbol{u}(t)} + \mu_{e}\nr{\boldsymbol{\nabla}\boldsymbol{u}(t)} & \leq M \nrm{c(t)} \|\boldsymbol{u}_2(t)\|_{L^4} \|\boldsymbol{u}(t)\|_{L^4}.
\end{align*}
Finally applying the Sobolev inequality followed by Young's inequality, we get 
\begin{align}\label{uniqueness:eq4}
 \frac{1}{2}\frac{\mathrm{d}}{\mathrm{d} t} \nr{\boldsymbol{u}(t)} + (\mu_{e} - \epsilon) \nr{\boldsymbol{\nabla}\boldsymbol{u}(t)}  \leq  M_{\epsilon} \|\boldsymbol{u}_2(t)\|_{L^4}^2 \nr{c(t)}.
 \end{align}
Adding equations \eqref{uniqueness:eq2} and \eqref{uniqueness:eq4}, we obtain 
\begin{align*}
    \frac{1}{2}\frac{\mathrm{d}}{\mathrm{d} t}\left( \nr{c(t)} + \nr{\boldsymbol{u}(t)} \right) + (D - \epsilon) \nr{\boldsymbol{\nabla}c(t)} + (\mu_e - \epsilon) \nr{\boldsymbol{\nabla}\boldsymbol{u}(t)}\leq \Phi(t) \left( \nr{c(t)} + \nr{\boldsymbol{u}(t)} \right),
\end{align*}
where $\Phi(t) =\left(M_\epsilon + \kappa_2 + M_{\epsilon} \|\boldsymbol{u}_2\|_{L^4}^2  \right)$.  We choose $\epsilon > 0$ satisfying $\epsilon < \min\left\{\frac{D}{2}, \frac{\mu_e}{2}\right\}$. Consequently, neglecting the non-negative terms and applying the Gr\"onwall's inequality, we arrive at
\begin{align*}
       \nr{c_1(t) - c_2(t)} + \nr{\boldsymbol{u}_1(t) - \boldsymbol{u}_2(t)} \leq \left( \nr{c_1(0) - c_2(0)} + \nr{\boldsymbol{u}_1(0) - \boldsymbol{u}_2(0)} \right) \exp{\left(2 \int_{0}^{T} \Phi(\tau) d \tau\right)},
\end{align*}
for all $t \in [0,T]$. This completes the proof for the continuous dependence of solutions $(c,{\bf u})$ on initial data. 
\end{proof}
\begin{thm}[Uniqueness]\label{uniqueness}
Let $(c_1, \boldsymbol{u}_1)$ and $(c_2, \boldsymbol{u}_2)$ be two solutions to the IBVP \eqref{model1}-\eqref{assumption:beta-k} subject to the same initial data. Then we have $   c_1(t,x) = c_2(t,x) \text{ and } \boldsymbol{u}_1(t,x) = \boldsymbol{u}_2(t,x)$, for almost every $(t,x) \in (0,T) \times \Omega$.
\end{thm}

\begin{proof}
Let $(c_1, \boldsymbol{u}_1)$ and $(c_2, \boldsymbol{u}_2)$ be two solutions to the system \eqref{model1}--\eqref{assumption:beta-k}.  Then using $c_1(0)=c_2(0)$ and ${\bf u}_1(0)={\bf u}_2(0)$ in  the continuous dependence Theorem \ref{Thm:continous dependence}, we have  
\begin{align*}
\big\|c_1(t) - c_2(t)\big\|_{L^2} + \big\|\boldsymbol{u}_1(t) - \boldsymbol{u}_2(t)\big\|_{L^2}= 0, \quad \mbox{for all} \  t \in [0,T].
\end{align*}
This immediately gives that $c_1(t,x) = c_2(t,x)$, and $\boldsymbol{u}_1(t,x) = \boldsymbol{u}_2(t,x)$ for almost every $(t,x)\in (0,T) \times \Omega$, concluding the proof of uniqueness. 
\end{proof}
\section{Existence, Uniqueness, and Continuous Dependence of Strong Solutions}\label{sec: strong solution}
\begin{thm} If  $(c_0,\boldsymbol{u}_0) \in H^1(\Omega)\times \boldsymbol{V}$ is such that   \( 0 \leq c_0  \leq 1 \), almost everywhere in \( \Omega \), then there exists a strong solution
for the IBVP ~\eqref{model1}--\eqref{assumption:beta-k}.     
\end{thm}    
\begin{proof}
For a fixed $\ell\in\mathbb{N}$,  take  $\{(c_{\ell, n}, \boldsymbol{u}_{\ell, n})\}_{n\in \mathbb{N}}\subset H^2(\Omega)\times\left(\boldsymbol{V}\cap \boldsymbol{H}^2(\Omega)\right)$ as the Galerkin approximation sequence for the truncated problem \eqref{model1}--\eqref{assumption:beta-k} as is done in section \ref{sec: Existence}. Using Lemmas \ref{lemma1} and \ref{lemma2}, we get that  $\{c_{\ell, n}\}_{n\in\mathbb{N}}$ and $\{\boldsymbol{u}_{\ell, n}\}_{n\in\mathbb{N}}$ are uniformly bounded in $L^2(0,T;H^1(\Omega))$ and $L^2(0,T;\boldsymbol{V})$ respectively. To derive higher regularity estimates for the velocity field, we substitute $\boldsymbol{v} = - \Delta \boldsymbol{u}_{\ell, n}$ into equation \eqref{finite weak 1}, which after using the integration by parts gives us
\begin{align*}
  \frac{1}{2} \frac{\mathrm{d}}{\mathrm{d} t} \nr{\boldsymbol{\nabla}\boldsymbol{u}_{\ell, n}(t)} + \mu_{e}\nr{\Delta \boldsymbol{u}_{\ell, n}(t)} = \br{\frac{\widetilde{\mu}_{\ell}(c_{\ell, n}(t))}{K}\boldsymbol{u}_{\ell, n}(t)}{\Delta \boldsymbol{u}_{\ell, n}(t)} +  \br{\beta |\boldsymbol{u}_{\ell, n}(t)|\boldsymbol{u}_{\ell, n}(t)}{\Delta \boldsymbol{u}_{\ell, n}(t)}   - \br{\boldsymbol{f}(t)}{\Delta \boldsymbol{u}_{\ell, n}(t)},
\end{align*}
for almost every $t \in (0, T).$ Using the bounds on $\widetilde{\mu}_{\ell}$ and $\beta$, together with Cauchy–Schwarz inequality, the above equation reduces to
\begin{align*}
\frac{1}{2} \frac{\mathrm{d}}{\mathrm{d} t} \nr{\boldsymbol{\nabla}\boldsymbol{u}_{\ell, n}(t)}   + \mu_{e}\nr{\Delta \boldsymbol{u}_{\ell, n}(t)} \leq   \frac{\mu_{\ell}}{K}  \nrm{\boldsymbol{u}_{\ell, n}(t)} \nrm{\Delta \boldsymbol{u}_{\ell, n}(t)} + \beta_2 \|\boldsymbol{u}_{\ell, n}(t)\|_{L^4}^2 \nrm{\Delta \boldsymbol{u}_{\ell, n}(t)} + \nrm{\boldsymbol{f}(t)}\nrm{\Delta \boldsymbol{u}_{\ell, n}(t)}.
\end{align*}
An application of Young's inequality gives
\begin{align*}
\frac{1}{2} \frac{\mathrm{d}}{\mathrm{d} t} \nr{\boldsymbol{\nabla}\boldsymbol{u}_{\ell, n}(t)} + \mu_{e}\,\nr{\Delta \boldsymbol{u}_{\ell, n}(t)} 
 &\leq 3 \epsilon \nr{\Delta \boldsymbol{u}_{\ell, n}(t)} 
   + M_1(\epsilon) \frac{\mu_{\ell}^2}{K^2}  \nr{\boldsymbol{u}_{\ell, n}(t)}   
   +  M_2(\epsilon) \beta_2^2 \|\boldsymbol{u}_{\ell, n}(t)\|_{L^4}^4 + M_3(\epsilon) \nr{\boldsymbol{f}(t)} .
\end{align*}
Defining $M_{\epsilon}:= \max\left\{ M_1(\epsilon) \frac{\mu_{\ell}^2}{K^2}, M_2(\epsilon) \beta_2^2, M_3(\epsilon) \right\}$. Now after applying the Gagliardo-Nirenberg and Poincar\'e inequalities, the above inequality reduces to 
\begin{align}\label{strong:eq1}
\frac{1}{2} \frac{\mathrm{d}}{\mathrm{d} t} \nr{\boldsymbol{\nabla}\boldsymbol{u}_{\ell, n}(t)} + (\mu_{e} - 3 \epsilon)\,\nr{\Delta \boldsymbol{u}_{\ell, n}(t)} 
 &\leq  M_{\epsilon}\left( \nr{\boldsymbol{f}(t)} +  \nr{\boldsymbol{u}_{\ell, n}(t)}   
   +   \|\boldsymbol{\nabla}\boldsymbol{u}_{\ell, n}(t)\|_{L^2}^4   \right).
\end{align}
Now to obtain higher regularity estimate for concentration,  substitute $\phi(t) = -\Delta c_{\ell, n}(t)$ in equation \eqref{finite weak 2} and use the integration by parts, to obtain
\begin{align*}
\frac{1}{2} \frac{\mathrm{d}}{\mathrm{d} t}\nr{\boldsymbol{\nabla}c_{\ell, n}(t)} + D \nr{\Delta c_{\ell, n}(t)} =  \br{\boldsymbol{u}_{\ell, n}(t)\cdot \boldsymbol{\nabla}c_{\ell, n}(t)}{\Delta c_{\ell, n}(t)} -  \br{\kappa \,c_{\ell, n}^2(t)}{\Delta c_{\ell, n}(t)}   +  \br{\kappa \,c_{\ell, n}(t)}{\Delta c_{\ell, n}(t)}.
\end{align*}
Applying Cauchy–Schwarz inequality, we get 
\begin{align*}
 \frac{1}{2} \frac{\mathrm{d}}{\mathrm{d} t}\nr{\boldsymbol{\nabla}c_{\ell, n}(t)} &+ D \nr{\Delta c_{\ell, n}(t)} \leq \|\boldsymbol{u}_{\ell, n}(t)\|_{L^4} \|\boldsymbol{\nabla} c_{\ell, n}(t) \|_{L^4} \nrm{\Delta c_{\ell, n}(t)}\\
 & \ \ + \kappa_2 \|c_{\ell, n}(t)\|_{L^{4}}^{2}\nrm{\Delta c_{\ell, n}(t)} +  \kappa_2 \nrm{c_{\ell, n}(t)}\nrm{\Delta c_{\ell, n}(t)}.
\end{align*}
After applying the Gagliardo-Nirenberg's inequality \ref{gagliardo} and the Sobolev inequality, we find that
\begin{align*}
 \frac{1}{2} \frac{\mathrm{d}}{\mathrm{d} t}\nr{\boldsymbol{\nabla}c_{\ell, n}(t)} + D \nr{\Delta c_{\ell, n}(t)}&\leq M_1\|\boldsymbol{\nabla}\boldsymbol{u}_{\ell, n}(t)\|_{L^2} \|\boldsymbol{\nabla} c_{\ell, n}(t) \|_{L^2}^{(4-d)/4}  \nrm{\Delta c_{\ell, n}(t)}^{(4+d)/4}\\
 &\quad + M_2 \kappa_2 \left(\nr{c_{\ell, n}(t)} +\|\boldsymbol{\nabla}c_{\ell, n}(t)\|_{L^2}^{2} \right)\nrm{\Delta c_{\ell, n}(t)} +  \kappa_2 \nrm{c(t)}\nrm{\Delta c_{\ell, n}(t)}.
\end{align*}
Finally, using the Young's inequality, we get 
\begin{align*}
\begin{aligned}
   \frac{1}{2} \frac{\mathrm{d}}{\mathrm{d} t}\nr{\boldsymbol{\nabla}c_{\ell, n}(t)} + (D - 3 \epsilon)\nr{\Delta c_{\ell, n}(t)}  &\leq   M_{1}(\epsilon)\Big(\|\boldsymbol{\nabla}\boldsymbol{u}_{\ell, n}(t)\|_{L^2}^{16/(4-d)} + \nrm{\boldsymbol{\nabla}c_{\ell, n}(t)}^4 \Big) \\
   &\quad + M_{2} (\epsilon) \kappa_2^{2} \left(\nrm{c_{\ell, n}(t)}^4 + \nrm{\boldsymbol{\nabla}c_{\ell, n}(t)}^4  \right) + M_{3} (\epsilon) \kappa_2^2\nr{c_{\ell, n}(t)}.
   \end{aligned} 
\end{align*}
Defining $M_{\epsilon }:= \max\{ M_{1} (\epsilon), M_{2} (\epsilon) \kappa_2^2, M_{3} (\epsilon) \kappa_2^2 \}$, we can rewrite the previous inequality as
\begin{align}\label{strong:eq2}
   &\frac{1}{2} \frac{\mathrm{d}}{\mathrm{d} t}\nr{\boldsymbol{\nabla}c_{\ell, n}(t)} + (D - 3 \epsilon) \nr{\Delta c_{\ell, n}(t)}\nonumber \\ & \qquad \leq   M_{\epsilon} \left( \nr{c_{\ell, n}(t)} + \nrm{c_{\ell, n}(t)}^4 + \nrm{\boldsymbol{\nabla}c_{\ell, n}(t)}^4    + \|\boldsymbol{\nabla}\boldsymbol{u}_{\ell, n}(t)\|_{L^2}^{16/(4-d)} + \nrm{\boldsymbol{\nabla}c_{\ell, n}(t)}^4 \right).
\end{align}
Adding the inequalities \eqref{strong:eq1} and \eqref{strong:eq2}, we get 
\begin{align}\label{strong:eq3}
\begin{aligned}
    &\frac{1}{2} \frac{\mathrm{d}}{\mathrm{d} t}\left(\nr{\boldsymbol{\nabla}c_{\ell, n}(t)} + \nr{\boldsymbol{\nabla}\boldsymbol{u}_{\ell, n}(t)}  \right) + (D - 3 \epsilon) \nr{\Delta c_{\ell, n}(t)} + (\mu_{e} - 3 \epsilon)\,\nr{\Delta \boldsymbol{u}_{\ell, n}(t)} \\ &\quad \leq M_{\epsilon} \left( \nr{c_{\ell, n}(t)}    + \nrm{c_{\ell, n}(t)}^4 + 2\nrm{\boldsymbol{\nabla}c_{\ell, n}(t)}^4    + \|\boldsymbol{\nabla}\boldsymbol{u}_{\ell, n}(t)\|_{L^2}^{16/(4-d)} +  \nr{\boldsymbol{f}(t)} +  \nr{\boldsymbol{u}_{\ell, n}(t)}   
   +   \|\boldsymbol{\nabla}\boldsymbol{u}_{\ell, n}(t)\|_{L^2}^4 \right).
   \end{aligned}
\end{align}
Choosing    $\epsilon < \min\left\{\frac{\mu_e}{3}, \frac{D}{3}\right\}$ along with dropping the non-negative terms in the left hand of above inequality and an application of  Gr\"onwall's lemma,  implies that there exists a $T^{\ast} \in (0,T)$, such that the sequences $\{c_{\ell, n}\}_{n \in \mathbb{N}}$ and $\{\boldsymbol{u}_{\ell, n}\}_{n \in \mathbb{N}}$ are uniformly bounded in $L^{\infty}(0, T^{\ast}; H^1(\Omega)) \cap L^2(0, T^{\ast}; H^2(\Omega))$ and $L^{\infty}(0, T^{\ast};\boldsymbol{V}) \cap L^2(0, T^{\ast};\boldsymbol{V} \cap \boldsymbol{H}^2(\Omega))$, respectively. Next using these estimates together with the triangle inequality in equations \eqref{model2} and \eqref{model3}, we get the sequences  $\left\{\frac{\partial c_{\ell, n}}{\partial t}\right\}_{n \in \mathbb{N}}$ and $\left\{\frac{\partial \boldsymbol{u}_{\ell, n}}{\partial t}\right\}_{n \in \mathbb{N}}$ are uniformly bounded in $L^2(0, T^{\ast}; L^2(\Omega))$ and $L^2(0, T^{\ast}; \boldsymbol{S})$, respectively. Following the arguments used in Lemma \ref{lemma3}, we pass to the limit as $n \to \infty$ to establish that the limit pair $(c_{\ell}, \boldsymbol{u}_{\ell})$ belongs to $L^2(0,T^*;H^2(\Omega)) \times L^2(0,T^*;\boldsymbol{V}\cap \boldsymbol{H}^2(\Omega))$ and constitutes a local strong solution to the truncated IBVP \eqref{model1}--\eqref{assumption:beta-k}. Next, invoking Lemma \ref{max lemma} and employing arguments analogous to those used in the proof of Lemma \ref{Global existence for truncated}, we extend this local solution globally to the interval $(0,T)$. Finally, adapting the arguments from proof of the Theorem \ref{thm:existence}, we conclude that the limit $(c,\boldsymbol{u})$ of the sequence $\{(c_\ell, \boldsymbol{u}_\ell)\}_{\ell \in \mathbb{N}}$ is a strong solution to the original IBVP \eqref{model1}--\ref{assumption:beta-k}.
\end{proof}


\begin{thm}[Continuous dependence on initial data]\label{Thm:continous dependence for strong solutions}
Let $(c_1,\boldsymbol{u}_1)$ and $(c_2,\boldsymbol{u}_2)$ be two strong solutions of the IBVP \eqref{model1}-\eqref{assumption:beta-k} with the initial data
$\left(c_{1}(0,\boldsymbol{x}), {\bf u}_1(0,x)\right)$ and $\left(c_{2}(0,\boldsymbol{x}), {\bf u}_2(0,x)\right)$, respectively. Then for all $t \in [0,T]$,
    \begin{align*}
       &\nr{c_1(t) - c_2(t) }+ \nr{\boldsymbol{\nabla}c_1(t)- \boldsymbol{\nabla}c_2(t)} + \nr{\boldsymbol{u}_1(t) - \boldsymbol{u}_2(t)} +\nr{\boldsymbol{\nabla}\boldsymbol{u}_1(t)-\boldsymbol{\nabla}\boldsymbol{u}_2(t)} \nonumber \\ &\  \ \leq  \Big( \nr{c_1(0) - c_2(0)}+ \nr{\boldsymbol{\nabla}c_1(0)- \boldsymbol{\nabla}c_2(0)} + \nr{\boldsymbol{u}_1(0) - \boldsymbol{u}_2(0)} +\nr{\boldsymbol{\nabla}\boldsymbol{u}_1(0)-\boldsymbol{\nabla}\boldsymbol{u}_2(0)}\Big) \exp{\left(2\int_{0}^{T} \varphi(\tau) d\tau \right)},
    \end{align*}
   for some nonnegative function $\varphi\in L^{1}(0,T)$ depending on $c_1,c_2$, ${\bf u}_{1},{\bf u}_{2}$, $\mu$ and $\kappa$. In particular, if $c_1(0)=c_2(0)$ and ${\bf u}_1(0)={\bf u}_{2}(0)$, we have \[c_1(t,x)=c_2(t,x)\ \mbox{and}\ {\bf u}_1(t,x)={\bf u}_2(t,x),\ \mbox{for}\ (t,x)\in (0,T)\times \Omega.\]
\end{thm}

\begin{proof}
We define the difference variables $c := c_1 - c_2$ and $\boldsymbol{u} := \boldsymbol{u}_1 - \boldsymbol{u}_2$. 
Then choosing the test function $\phi = - \Delta c$ in \eqref{uniqueness:eq1}, we obtain
\begin{align*}   
\frac{1}{2}\frac{\mathrm{d}}{\mathrm{d} t}\nr{\boldsymbol{\nabla}c(t)} - \br{\boldsymbol{u}_1(t)\cdot \boldsymbol{\nabla}c_1(t)- \boldsymbol{u}_2(t)\cdot \boldsymbol{\nabla}c_2(t)}{\Delta c(t)} + D \nr{\Delta c(t)} = - \br{\kappa(c_1(t) + c_2(t) -1)c(t)}{\Delta c(t)}.
\end{align*}
Rearranging the second term on the left-hand side of the above equation, we have
\begin{align*}
   \frac{1}{2}\frac{\mathrm{d}}{\mathrm{d} t}\nr{\boldsymbol{\nabla}c(t)}  + D \nr{\Delta c(t)} = \br{\boldsymbol{u}_1(t)\cdot \boldsymbol{\nabla}c(t)}{\Delta c(t)} + \br{\boldsymbol{u}(t)\cdot \boldsymbol{\nabla}c_2(t)}{\Delta c(t)} - \br{\kappa(c_1(t) + c_2(t) -1)c(t)}{\Delta c(t)}.
\end{align*}
Applying Cauchy–Schwarz inequality and $ 0 \leq c_1, c_2 \leq 1$ a.e. in $(0,T) \times \Omega$, we get
\begin{align*}
   \frac{1}{2}\frac{\mathrm{d}}{\mathrm{d} t}\nr{\boldsymbol{\nabla}c(t)}  + D \nr{\Delta c(t)} \leq  \|\boldsymbol{u}_1(t)\|_{L^4} \|\boldsymbol{\nabla}c(t)\|_{L^4} \nrm{\Delta c(t)} + \|\boldsymbol{\nabla}c_2(t)\|_{L^4} \|\boldsymbol{u}(t)\|_{L^4} \nrm{\Delta c(t)} + \kappa_2 \nrm{c(t)} \nrm{\Delta c(t)}.
\end{align*}
Using the Gagliardo-Nirenberg inequality \ref{gagliardo} and the Sobolev inequalities, we obtain
\begin{align*}
   \frac{1}{2}\frac{\mathrm{d}}{\mathrm{d} t}\nr{\boldsymbol{\nabla}c(t)}  + D \nr{\Delta c(t)}&\leq  M_1\|\boldsymbol{u}_1(t)\|_{L^4} \|\boldsymbol{\nabla}c(t)\|_{L^2}^{(4-d)/4} \nrm{\Delta c(t)}^{(4+d)/4} \\ & \qquad  + M_2\|\boldsymbol{\nabla}c_2(t)\|_{L^4} \|\boldsymbol{\nabla}\boldsymbol{u}(t)\|_{L^2} \nrm{\Delta c(t)} + \kappa_2 \nrm{c(t)} \nrm{\Delta c(t)}.
\end{align*}
An application of Young's inequality \ref{youngs inequaity} gives 
\begin{align}\label{strong uniqueness:eq5}
\begin{aligned}
   \frac{1}{2}\frac{\mathrm{d}}{\mathrm{d} t}\nr{\boldsymbol{\nabla}c(t)}  + (D - 3 \epsilon) \nr{\Delta c(t)} \leq  M_1(\epsilon)\|\boldsymbol{u}_1(t)\|_{L^4}^{8/(4-d)} \|\boldsymbol{\nabla}c(t)\|_{L^2}^{2}  + M_2(\epsilon)\|\boldsymbol{\nabla}c_2(t)\|_{L^4}^{2} \|\boldsymbol{\nabla}\boldsymbol{u}(t)\|_{L^2}^{2} + M_3(\epsilon) \kappa_2^2 \nr{c(t)}. 
   \end{aligned}
\end{align}
Next, choosing  $\boldsymbol{v} = -\Delta \boldsymbol{u}$, in equation \eqref{uniqueness:eq2}, we get 
\begin{align*}
    \dfrac{1}{2}\frac{\mathrm{d}}{\mathrm{d} t}\nr{\boldsymbol{\nabla}\boldsymbol{u}(t)} + \mu_e \nr{\Delta \boldsymbol{u}(t)} = \Big((\mu(c_1(t)) \boldsymbol{u}_1(t) - \mu(c_2(t)) \boldsymbol{u}_2(t)), \Delta \boldsymbol{u}(t) \Big) + \Big(\beta(|\boldsymbol{u}_1(t)|\boldsymbol{u}_1(t) - |\boldsymbol{u}_2(t)|\boldsymbol{u}_2(t)) , \Delta \boldsymbol{u}(t)  \Big).
\end{align*}
Rearranging the terms on the right-hand side of the above inequality, we have
\begin{align*}
   \dfrac{1}{2}\frac{\mathrm{d}}{\mathrm{d} t}\nr{\boldsymbol{\nabla}\boldsymbol{u}(t)} + \mu_e \nr{\Delta \boldsymbol{u}(t)} &= \Big((\mu(c_1(t)) \boldsymbol{u}(t), \Delta \boldsymbol{u}(t) \Big) + \Big((\mu(c_1(t)) - \mu(c_2(t)) )\boldsymbol{u}_2(t), \Delta \boldsymbol{u}(t) \Big)\\
&\qquad + \Big(\beta|\boldsymbol{u}_1(t)|\boldsymbol{u}(t) , \Delta \boldsymbol{u}(t)  \Big)   +\Big(\beta(|\boldsymbol{u}_1(t)| - |\boldsymbol{u}_2(t)|)\boldsymbol{u}_2(t) , \Delta \boldsymbol{u}(t)\Big).
\end{align*}
The Cauchy–Schwarz inequality along with the local Lipschitz continuity of $\mu$ and the boundedness of $c_1, c_2, \beta$, yields
\begin{align*}
\dfrac{1}{2}\frac{\mathrm{d}}{\mathrm{d} t}\nr{\boldsymbol{\nabla}\boldsymbol{u}(t)} + \mu_e \nr{\Delta \boldsymbol{u}(t)}&\leq \mu_1 \nrm{\boldsymbol{u}(t)}  \nrm{\Delta \boldsymbol{u}(t)}  + M \|c(t)\|_{L^4} \|\boldsymbol{u}_2(t)\|_{L^4} \nrm{\Delta \boldsymbol{u}(t)} \\ &\qquad + \beta_2 \big(\|\boldsymbol{u}_1(t) \|_{L^4} + \|\boldsymbol{u}_2(t)\big)\|\boldsymbol{u}(t) \|_{L^4} \nrm{\Delta \boldsymbol{u}(t)}. 
\end{align*}
After applying the Sobolev inequality followed by Young’s inequality, we obtain
\begin{align*}
  \dfrac{1}{2}\frac{\mathrm{d}}{\mathrm{d} t}\nr{\boldsymbol{\nabla}\boldsymbol{u}(t)} + (\mu_e - 4 \epsilon) \nr{\Delta \boldsymbol{u}(t)}&\leq  M_1(\epsilon) \nr{u(t)} + M_2(\epsilon) \|\boldsymbol{u}_2(t)\|_{L^4}^{2} \|c(t)\|_{H^1}^2 \\ &\qquad + M_3(\epsilon) \|\boldsymbol{u}_1(t)\|_{L^4}^{2} \nr{\boldsymbol{u}(t)}  + M_4(\epsilon) \|\boldsymbol{u}_2(t)\|_{L^4}^{2} \nr{\boldsymbol{\nabla}\boldsymbol{u}(t)}.
\end{align*}
Define $M_\epsilon := \max\{M_1(\epsilon), M_2(\epsilon), M_3(\epsilon), M_4(\epsilon)\}$. The preceding inequality then implies
\begin{align}\label{strong uniqueness:eq6}
   &\dfrac{1}{2}\frac{\mathrm{d}}{\mathrm{d} t}\nr{\boldsymbol{\nabla}\boldsymbol{u}(t)} + (\mu_e - 4 \epsilon) \nr{\Delta \boldsymbol{u}(t)} \nonumber \\ & \qquad \leq  M_{\epsilon} \left( \nr{u(t)} +  \|\boldsymbol{u}_2(t)\|_{L^4}^{2} \|c(t)\|_{H^1}^2 +  \|\boldsymbol{u}_1(t)\|_{L^4}^{2} \nr{\boldsymbol{u}(t)} + \|\boldsymbol{u}_2(t)\|_{L^4}^{2} \nr{\boldsymbol{\nabla}\boldsymbol{u}(t)}\right).
\end{align}
Adding inequalities \eqref{uniqueness:eq2}, \eqref{uniqueness:eq4},\eqref{strong uniqueness:eq5} and \eqref{strong uniqueness:eq6}, we get 
\begin{align*}
 &\dfrac{1}{2}\frac{\mathrm{d}}{\mathrm{d} t} \big(\nr{c(t)}+ \nr{\boldsymbol{\nabla}c(t)} + \nr{\boldsymbol{u}(t)} +\nr{\boldsymbol{\nabla}\boldsymbol{u}(t)} \big)+ (D - \epsilon)\nr{\boldsymbol{\nabla}c(t)} + (D - 3 \epsilon) \nr{\Delta c(t)} \\ & + (\mu_{e} - \epsilon) \nr{\boldsymbol{\nabla}\boldsymbol{u}(t)}  + (\mu_e - 4 \epsilon) \nr{\Delta \boldsymbol{u}(t)} \leq \varphi(t) \left(\nr{c(t)}+ \nr{\boldsymbol{\nabla}c(t)} + \nr{\boldsymbol{u}(t)} + \nr{\boldsymbol{\nabla}\boldsymbol{u}(t)} \right),
\end{align*}
where $\varphi(t) := M_{\epsilon}\left(\|\boldsymbol{u}_1(t)\|_{L^4}^2 + \|\boldsymbol{u}_2(t)\|_{L^4}^2 +  \|\boldsymbol{u}_1(t)\|_{L^4}^{8/(4-d)}  + \|\boldsymbol{\nabla}c_2(t)\|_{L^4}^{2}\right).$ Now we choose $\epsilon>0$ such that $\epsilon < \min\{D/3, \mu_e/4\}$ and ignoring nonnegtive  terms in  the left hand side of above inequality and then applying the Grownwall's inequality, we get 
\begin{align}\label{continuous dependence of strong solutions}
&\nr{c(t)}+ \nr{\boldsymbol{\nabla}c(t)} + \nr{\boldsymbol{u}(t)} +\nr{\boldsymbol{\nabla}\boldsymbol{u}(t)} \nonumber \\ & \qquad \qquad \leq  \Big( \nr{c(0)}+ \nr{\boldsymbol{\nabla}c(0)} + \nr{\boldsymbol{u}(0)} +\nr{\boldsymbol{\nabla}\boldsymbol{u}(0)}\Big) \exp{\left(2\int_{0}^{T} \varphi(\tau) d\tau \right)},
\end{align}
for all $t \in (0,T)$. This completes the proof of the continuous dependence of strong solutions on initial data. Now using $c_1(0,x) =c_2(0,x)$ and $\boldsymbol{u}_1(0,x) = \boldsymbol{u}_2(0,x)$ in inequality \ref{continuous dependence of strong solutions}, we conclude the uniqueness of  strong solutions.
\end{proof}

\section{Long-time Behavior}\label{sec: Asymptotic Behaviour}
In this section, we shall discuss the long-time behavior of the solution to the IBVP \eqref{model1}--\eqref{assumption:beta-k}. In particular, we shall focus on establishing the exponential decay and finite-time blow-up for the concentration. 
\subsection{Exponential decay of concentration}
\begin{thm}[Exponential decay of concentration]\label{thm: exponential decay}
Let $(c,\boldsymbol{u})$ be a weak solution of the system \eqref{model1}-\eqref{assumption:beta-k}  with initial data $c(0, \boldsymbol{x}):=c_0(\boldsymbol{x})$, for $x\in \Omega$. Suppose the initial concentration satisfies
\[
0 \leq c_0(\boldsymbol{x}) \leq M_0 < 1 \quad \text{a.e. in } \Omega.
\]
Then, for any $p \in [1, \infty]$, the $L^p$-norm of the concentration decays exponentially with a uniform rate $\lambda = \kappa_1(1-M_0)$:
\begin{align}\label{Lp decay}
  \|c(t)\|_{L^p(\Omega)} \leq C_p \, e^{-\lambda \,t}, \quad \mbox{for all} \ \  t \geq 0  
\end{align}
where the constant $C_p$ depends on $p$, $\Omega$ and the initial data $c_0$.
\end{thm}

\begin{proof}
We begin by proving the exponential decay in $L^1$ and $L^\infty$ norms separately, and then use the $L^p$-interpolation to conclude the proof of \eqref{Lp decay}. To prove the  $L^1$ decay,  take $\phi = 1$ in equation \eqref{weak2} and using integration parts, we get  that
\begin{align*}
\frac{\mathrm{d}}{\mathrm{d} t} \int_{\Omega }c(t) \mathrm{d} x  &= \int_{\Omega } \kappa\, c(c-1)\mathrm{d} x,\ \mbox{for all}\ t\geq 0. 
\end{align*}
By Lemma \ref{max lemma}, we have $0 \leq c(t,\boldsymbol{x}) \leq M_0 < 1$ and the assumption that $\kappa \geq \kappa_1$, we get
 $$\frac{\mathrm{d}}{\mathrm{d} t} \|c(t)\|_{L^1} \leq -\lambda \|c(t)\|_{L^1},\ \ \mbox{for all}\ t\geq 0$$ where $\lambda:=\kappa_1(1-M_0)$. 
After applying Gr\"onwall's inequality, we get 
\[
\|c(t)\|_{L^1(\Omega)} \leq \|c_0\|_{L^1(\Omega)} \, e^{- \lambda\,t},\ \mbox{for all}\ t\geq 0.
\]
Next, we prove the $L^\infty$ decay for the concentration. Define  $w(t, \boldsymbol{x}) := e^{\lambda\,t}\, c(t, \boldsymbol{x})$ and using this  in equation \ref{weak2}, to obtain  
\[
\innerproduct{\frac{\partial w(t)}{\partial t}, \phi} + \br{\boldsymbol{u}\cdot \boldsymbol{\nabla}w(t)}{\phi} + D\br{\boldsymbol{\nabla}w(t)}{\boldsymbol{\nabla} \phi} \leq 0, \quad \mbox{for all} \  \phi \in H^1(\Omega), \ \ \mbox{satisfying} \ \phi \geq 0, \ \mbox{for all}\ t\geq 0. 
\]
Now we will show that the maximum principle stated in Lemma \ref{max lemma} holds for $w$ defined above. 
We first  define $\widetilde{w}(t,x) := \max\{0, w(t,x) - \|w(0)\|_{L^{\infty}}\}$ for $(t,x)\in (0,\infty)\times \Omega$ and substituting $\phi = \widetilde{w}$, in the above inequality to get 
\[
\innerproduct{\frac{\partial w(t)}{\partial t}, \widetilde{w}(t)} + \br{\boldsymbol{u}\cdot \boldsymbol{\nabla}w(t)}{\widetilde{w}(t)} + D\br{\boldsymbol{\nabla}w(t)}{\boldsymbol{\nabla} \widetilde{w}(t)} \leq 0,\ \ \mbox{for all}\ t\geq 0. 
\]
Now arguing as in the proof of Lemma \ref{max lemma}, we conclude that $\|w(t)\|_{L^{\infty}(\Omega)} \leq \|w(0)\|_{L^{\infty}(\Omega)} = \|c_0\|_{L^{\infty}(\Omega)}$. This, with the  definition of $w,$ gives 
\[
\|c(t)\|_{L^{\infty}(\Omega)} \leq \|c_0\|_{L^{\infty}(\Omega)} \, e^{- \lambda\,t}, \ \mbox{for all}\ t\geq 0. 
\]
Thus, we have shown the decay of $c$ in both $L^1$ and $L^{\infty}$ norms. Now for any  $p \in (1, \infty)$, we apply the standard estimates on  $L^p$-norms of $c$ along with $L^1$ and $L^{\infty}$ estimates on $c$, to get
\begin{align*}
\|c(t)\|_{L^p} \leq 
 C_p \, e^{-\lambda t}, \ \mbox{for all}\ t\geq 0.
\end{align*}
where $C_p = \|c_0\|_{L^1(\Omega)}^{1/p} \|c_0\|_{L^{\infty}(\Omega)}^{1-1/p}$. This completes the proof of Theorem \ref{thm: exponential decay}.
\end{proof}

\subsection{Finite-Time Blow-Up}\label{sec:Blow-Up}
\begin{thm}\label{thm:blowup}
Let the initial concentration satisfy $c_0 \in L^\infty(\Omega)$ with $c_0(x) \geq M > 1$ almost everywhere in $\Omega$. Then, the solution $c(t,x)$ will blow up in  finite-time i.e.  there exists  $T^* < \infty$ such that
\[
\lim_{t \to T^{*}} \|c(t)\|_{L^\infty(\Omega)} = +\infty.
\]
\end{thm}
To establish the proof of Theorem \ref{thm:blowup}, we need to prove a lemma related to preserving the lower bound for concentration. We will do this in the following lemma. 
\begin{lem}\label{lem:blowup}
If  $c(0,x):=c_0(x) \geq M > 1$ for almost every $x \in \Omega$ then $c(t,x) \geq M,  \  \text{for a.e. } x \in \Omega \text{ and  } t \in [0, T].$ 
\end{lem}
\begin{proof}
Define $\widetilde{c}(t,x) := \max\{0, M-c(t,x)\}$ for $(t,x)\in (0,T)\times\Omega$ and  substitute $\phi = \widetilde{c}$,  in  equation  \eqref{weak2}, to get
    \begin{align*}
   - \frac{1}{2}\frac{\mathrm{d}}{\mathrm{d} t} \nr{\widetilde{c}}  - D \nr{\boldsymbol{\nabla}\widetilde{c}} =  \int_{\{ c \leq M\}} \kappa\, c (c-1) (M-c) \mathrm{d} x.
    \end{align*}
After adding and subtracting $M$ in the term $(c-1)$ of above equation,  we get 
\begin{align*}
\begin{aligned}
    \frac{1}{2}\frac{\mathrm{d}}{\mathrm{d} t} \nr{\widetilde{c}} +  D \nr{\boldsymbol{\nabla}\widetilde{c}} &=  \int_{\{ c \leq M\}} \kappa\, c (c- M ) (c-M) \mathrm{d} x + \int_{\{ c \leq M\}} \kappa\, c (M- 1 ) (c-M) \mathrm{d} x\\
    &\leq \int_{\{ c \leq M\}} \kappa\, c (c- M ) (c-M) \mathrm{d} x.
    \end{aligned}
    \end{align*}
Next after applying the bounds $0 < \kappa_1 \leq \kappa \leq \kappa_2$, we arrive at 
\begin{align*}
    \frac{1}{2}\frac{\mathrm{d}}{\mathrm{d} t} \nr{\widetilde{c}}  + D \nr{\boldsymbol{\nabla}\widetilde{c}} \leq \kappa_2\, M\nr{\widetilde{c}}.
\end{align*}
Finally, using the Gr\"onwall's inequality, we get
   \begin{align*}
     \nr{\widetilde{c}(t)}  \leq \nr{\widetilde{c}(0)}\, e^{2 \kappa_2\, M }, \text{ for } t \in (0,T).
   \end{align*} 
Utilizing $\widetilde{c}(0,x) = 0$, we conclude that $\widetilde{c}(t,x) = 0$ for all $(t,x) \in (0,T)\times\Omega$. Therefore, $c(t,x) \geq M$ a.e. in $(0,T) \times \Omega$.
\end{proof}

\begin{proof}[\textbf{Proof of Theorem}~\ref{thm:blowup}]
Define \begin{align}\label{Energy}
 \mathcal{M}(t) := \int_{\Omega}c(t,x) \mathrm{d} x.   
\end{align}
Differentiating $\mathcal{M}(t)$ with respect to $t$ and using convection diffusion reaction equation \eqref{model3}, we get
\begin{align}
 \mathcal{M}'(t)   = \int_{\Omega} \frac{\partial c(t)
 }{\partial t}\mathrm{d} x = \int_{\Omega} \Big( -\boldsymbol{u}\cdot\boldsymbol{\nabla}c + D \Delta c + \kappa\, c(c-1) \Big)\mathrm{d} x.
\end{align}
Using the integration by parts along with the boundary and incompressibility condition, Lemma \ref{lem:blowup} and the bound $\kappa \geq \kappa_1$, we get 
\begin{align*}
 \mathcal{M}'(t)  \geq \kappa_1 \left(\int_{\Omega}  c^2(t,x)\mathrm{d} x -  \int_{\Omega}  c(t,x) \mathrm{d} x \right).
\end{align*}
Apply the Cauchy–Schwarz inequality along with the definition of $\mathcal{M}(t)$, to get
\begin{align*} 
\mathcal{M}'(t) \geq 
\frac{\kappa_1}{|\Omega|} \left(\int_{\Omega} c\,\mathrm{d} x \,\right)^2 - \kappa_1 \int_{\Omega} c \,\mathrm{d} x=\frac{\kappa_1}{|\Omega|} \left(\mathcal{M}(t)^2 - |\Omega|\mathcal{M}(t) \right). 
\end{align*}
After solving the above differential inequality, we get 
\begin{align*} \mathcal{M}(t) \geq \frac{|\Omega|\mathcal{M}(0)}{\mathcal{M}(0) - (\mathcal{M}(0) - |\Omega|)e^{\kappa_1 t}}.
\end{align*}
Next using the relation $\|c(t)\|_{L^{\infty}(\Omega)} \geq |\Omega|^{-1}\mathcal{M}(t)$, we obtain 
\begin{align*}\|c(t)\|_{L^{\infty}(\Omega)} \geq \frac{\mathcal{M}(0)}{\mathcal{M}(0) - (\mathcal{M}(0) - |\Omega|)e^{\kappa_1 t}}.
\end{align*}
Now since the denominator of the right hand side of above inequality vanishes when $t=T^*:=\frac{1}{\kappa_1} \ln{\left(\frac{\mathcal{M}(0)}{\mathcal{M}(0)-|\Omega|}\right)}$, hence we get that 
\begin{align*}\lim_{t \to T^{\ast}} \|c(t)\|_{L^\infty(\Omega)} = +\infty.
\end{align*}
This completes the proof of Theorem \ref{thm:blowup}.
\end{proof}

\section{Numerical Validation}\label{Numerical Validation}

We utilize \textsc{COMSOL Multiphysics} \textsuperscript{\textregistered}  \cite{COMSOL2024} to numerically simulate the IBVP \eqref{model1}-\eqref{assumption:beta-k} where our goal is to validate the analytical decay and finite time blow up of concentration proved in section \ref{sec: Asymptotic Behaviour}.  The simulation is performed on the space-time cylinder $(0,T)\times \Omega$, where the spatial domain is defined as $\Omega = (0, 400) \times (0, 200)$. Adhering to the constraints required to establish the well-posedness analysis of IBVP \eqref{model1}-\eqref{assumption:beta-k}, throughout this section, we assume that the viscosity function $\mu(c) = e^{R\, c}$, where $R$ denotes the viscosity contrast coefficient.  
\subsection{Weak Formulation and Setup}
Let $\boldsymbol{V}$, $Q = L^2_0(\Omega) $, and $W = H^1(\Omega)$ denote the test function spaces for velocity, pressure and concentration respectively. The variational formulation is: Find $(\boldsymbol{u}, p, c) \in \boldsymbol{V} \times Q \times W$ such that the following equations
\begin{subequations}\label{weak_form}
    \begin{align}
        0 &= \int_{\Omega} (\nabla \cdot \boldsymbol{u}) \, q \, \mathrm{d} x, \\
        0 &= \int_{\Omega} \left[ \frac{\partial \boldsymbol{u}}{\partial t} \cdot \boldsymbol{v} + e^{R\,c} \boldsymbol{u} \cdot \boldsymbol{v} + |\boldsymbol{u}| \boldsymbol{u} \cdot \boldsymbol{v} + \boldsymbol{\nabla} \boldsymbol{u} : \boldsymbol{\nabla} \boldsymbol{v} - p \boldsymbol{\nabla}\cdot \boldsymbol{v} \right] \mathrm{d} x, \\
        0 &= \int_{\Omega} \left[ \frac{\partial c}{\partial t} \phi + (\boldsymbol{u} \cdot \nabla c) \phi + D \boldsymbol{\nabla} c \cdot \boldsymbol{\nabla} \phi - \kappa\, c(c-1) \phi \right] \mathrm{d} x
    \end{align}
\end{subequations}
hold for all $(q,\boldsymbol{v}, \phi) \in  Q \times \boldsymbol{V} \times W$. 
 Set the initial velocity $\boldsymbol{u}_0 = (0.1, 0)$ and  initial concentration distribution by 
\begin{align}
    c_0(x, y) =  \begin{cases} 
M_0, & 50 \leq x \leq 150, \\ 
0, & \text{otherwise},
\end{cases}
\end{align}
where $M_0 \in (0,1)$. In the numerical simulations, we vary the reaction coefficient $\kappa$ in the range $[0.005,\,0.02]$ and take $M_0 \in [0.4,\,0.8]$, while fixing $D = 0.005$ and $R = 1$.

\subsection{Decay Rate Analysis}
\begin{figure}[h!]
    \centering
    \begin{minipage}[t]{0.48\linewidth}
        \centering
        \includegraphics[width=1\linewidth, trim={0 210 10 210}, clip]{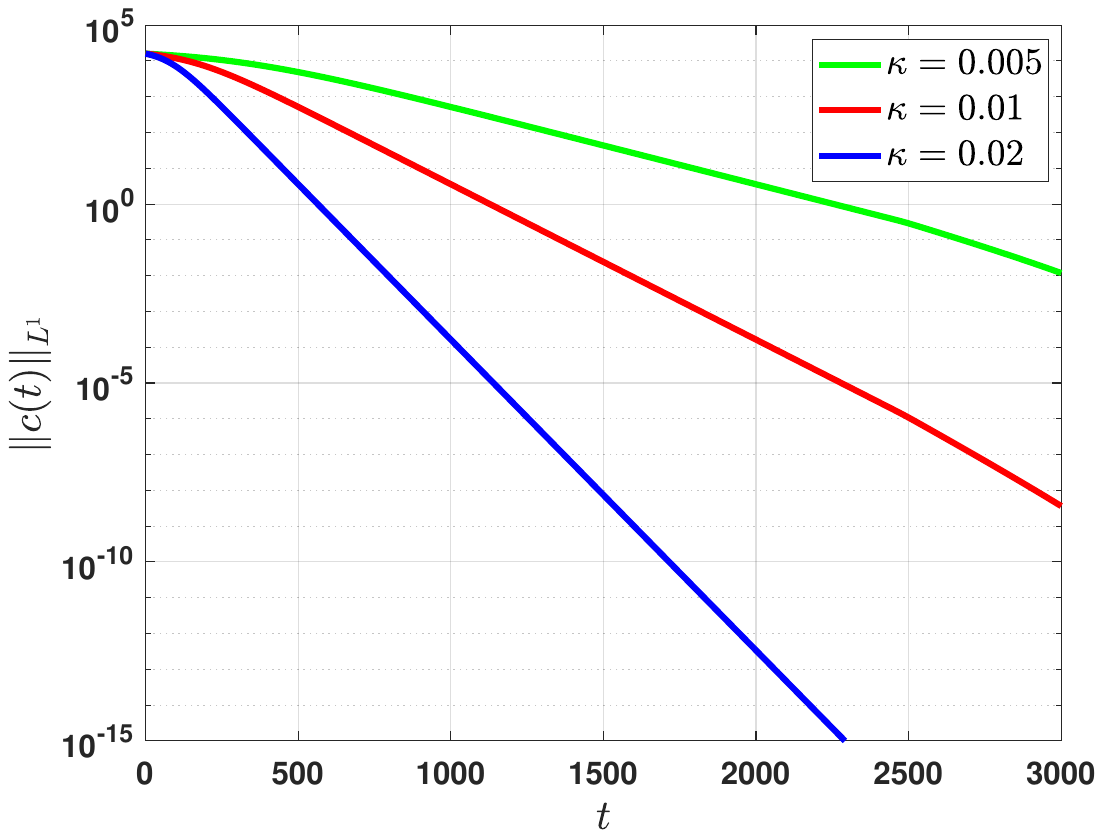}
        \caption{Time evolution of the $L^1(\Omega)$-norm of the concentration, $\|c(t)\|_{L^1(\Omega)}$, for different values of the reaction parameter $\kappa$, with $M_0 = 0.8$.}
        \label{fig: kappa effect on L_1 norm}
    \end{minipage}
    \hfill
    \begin{minipage}[t]{0.48\linewidth}
        \centering
        \includegraphics[width=1\linewidth, trim={0 210 10 210},clip]{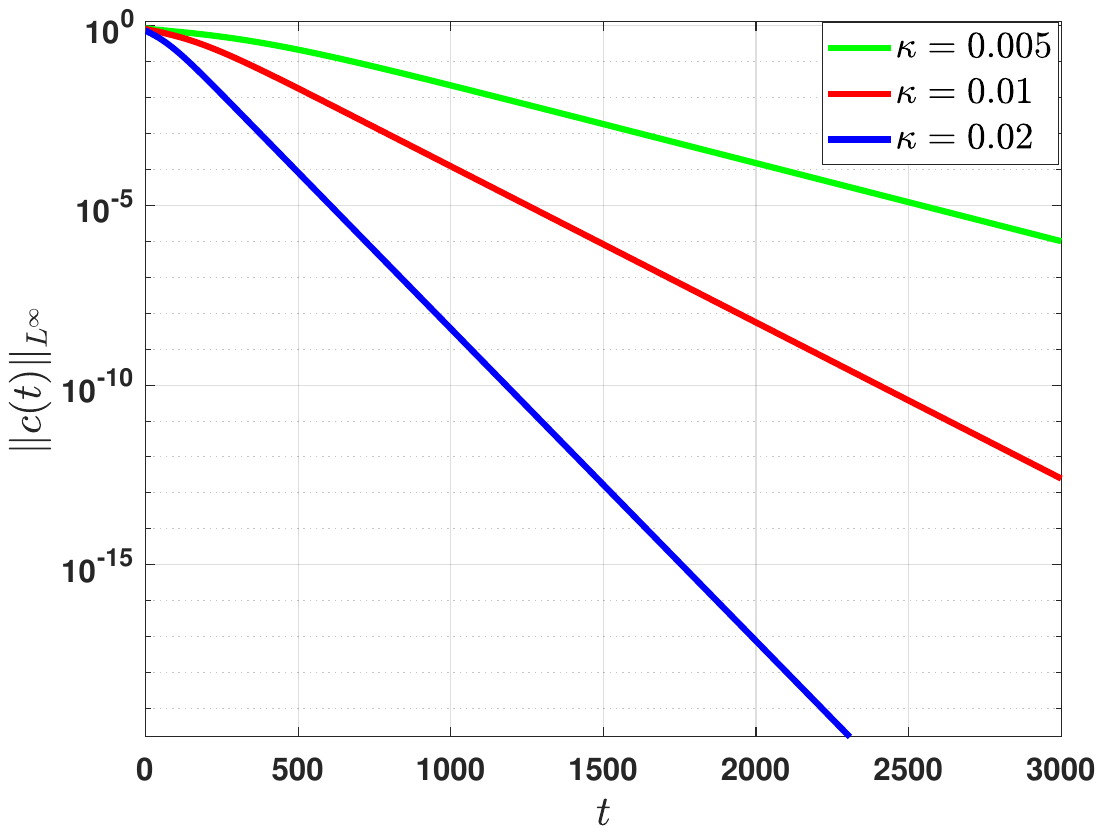}
        \caption{Time evolution of the $L^1(\Omega)$-norm of the concentration, $\|c(t)\|_{L^\infty(\Omega)}$, for different values of the reaction parameter $\kappa$, with $M_0 = 0.8$.}
        \label{fig: kappa effect on L_infty norm}
    \end{minipage}
\end{figure}
From Lemma \ref{thm: exponential decay}, we have  $\|c(t)\|_{L^p(\Omega)} \leq M e^{-\kappa(1-M_0)\, t}$, for $1\leq p\leq \infty$, which  indicates a theoretical minimum decay rate of $\lambda = \kappa(1-M_0)$. To verify this behavior, we plotted the time evolution of the $L^1$ and $L^\infty$ norms of concentration on a logarithmic scale (Figures \ref{fig: kappa effect on L_1 norm} and \ref{fig: kappa effect on L_infty norm}). The resulting plots exhibit a dominant linear trend characteristic of exponential decay, with slopes that steepen as $\kappa$ increases, consistent with the theoretical prediction. Closer inspection, however, reveals that these slopes are time-dependent, implying that the numerical decay rate $\lambda_{\mathrm{num}}$ is not strictly constant however evolves over time.
\begin{figure}[h!]
    \centering
    \begin{minipage}[t]{0.48\linewidth}
        \centering
        \includegraphics[width=1\linewidth, trim={0 210 10 210}, clip]{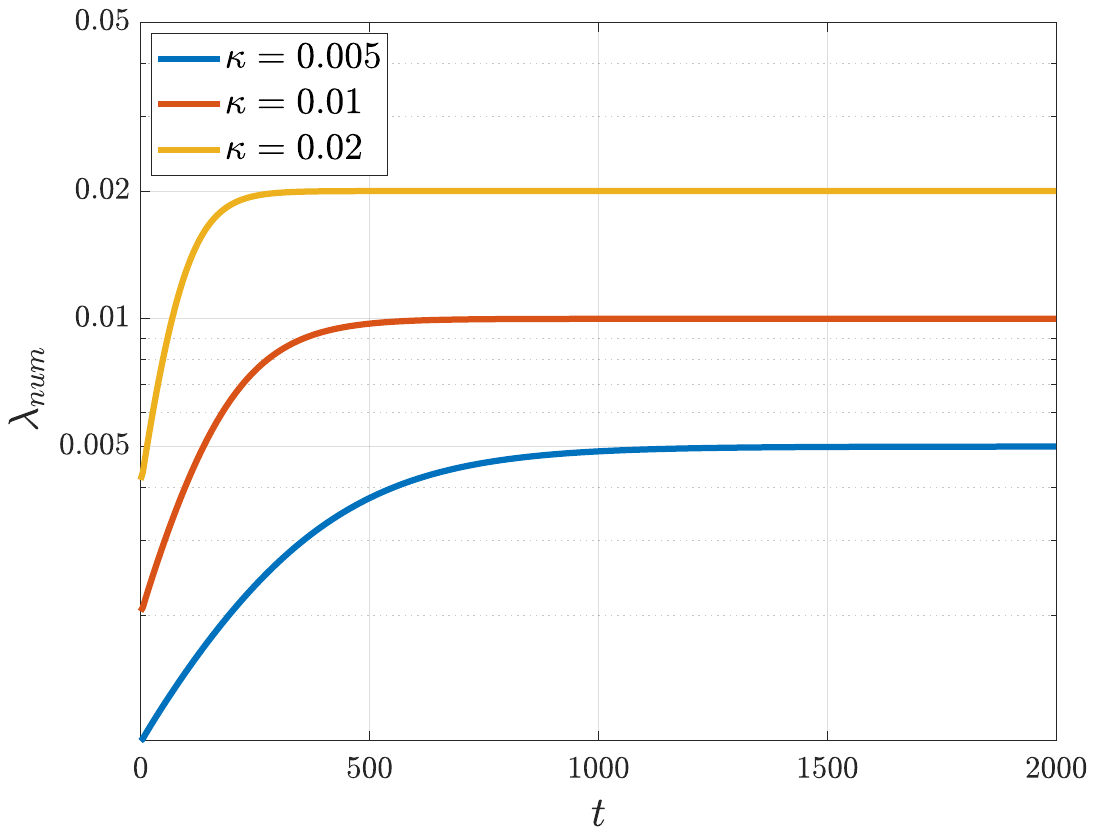}
        \caption{Numerical decay rate $\lambda_{\mathrm{num}}$ of the $L^1(\Omega)$-norm of the concentration, $\|c(t)\|_{L^1(\Omega)}$, for fixed $M_0 = 0.8$ and $\kappa = 0.005,\,0.01,\,0.02$.}
        \label{fig: decay rate}
    \end{minipage}
    \hfill
    \begin{minipage}[t]{0.48\linewidth}
        \centering
        \includegraphics[width=1\linewidth, trim={0 210 10 210},clip]{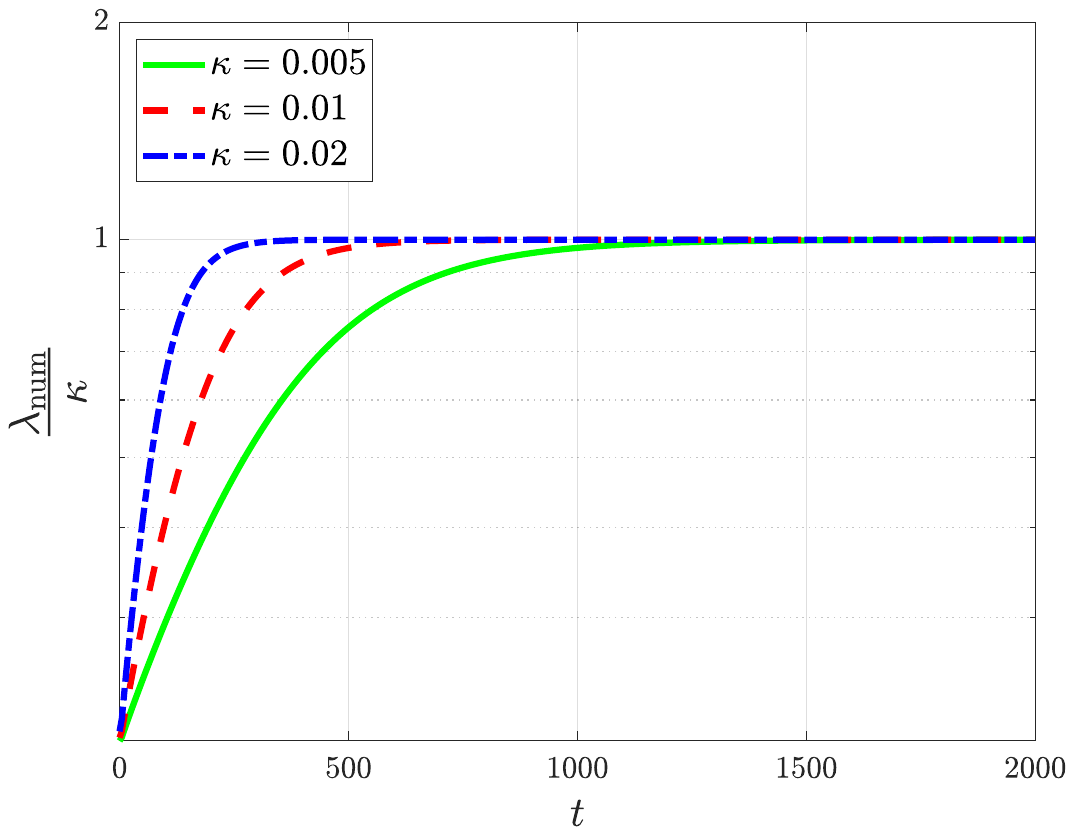}
        \caption{Numerical decay rate of the $L^1(\Omega)$-norm of the concentration, $\|c(t)\|_{L^1(\Omega)}$, for fixed $M_0 = 0.8$.}
        \label{fig: normalized decay rate}
    \end{minipage}
\end{figure}
In the proof of Lemma \ref{thm: exponential decay}, we employed the estimate $\kappa\, c(c-1) \leq - \kappa (1 - M_0) c $ while this holds globally, the asymptotic behavior as $c \to 0$ satisfies $\kappa\, c(c-1) \approx - \kappa\, c$. This suggests a transition in the decay rate: initially $\kappa (1 - M_0)$, eventually increasing to approach $\kappa$. To validate this, we calculated the numerical decay rate $\lambda_{\mathrm{num}} = - \frac{\mathrm{d}}{\mathrm{d}t} \log(\|c(t)\|_{L^1(\Omega)}$) for $\kappa \in \{0.005, 0.01, 0.02\}$. As shown in Figure \ref{fig: decay rate}, $\lambda_{\mathrm{num}}$ increases over time and eventually stabilizes. Furthermore, Figure \ref{fig: normalized decay rate} presents the decay rate normalized by $\kappa$; this confirms that the numerical decay rate asymptotically approaches the theoretical limit $\kappa$.
\begin{figure}[h!]
    \centering
    \begin{minipage}[t]{0.48\linewidth}
        \centering
        \includegraphics[width=1\linewidth, trim={0 210 10 210}, clip]{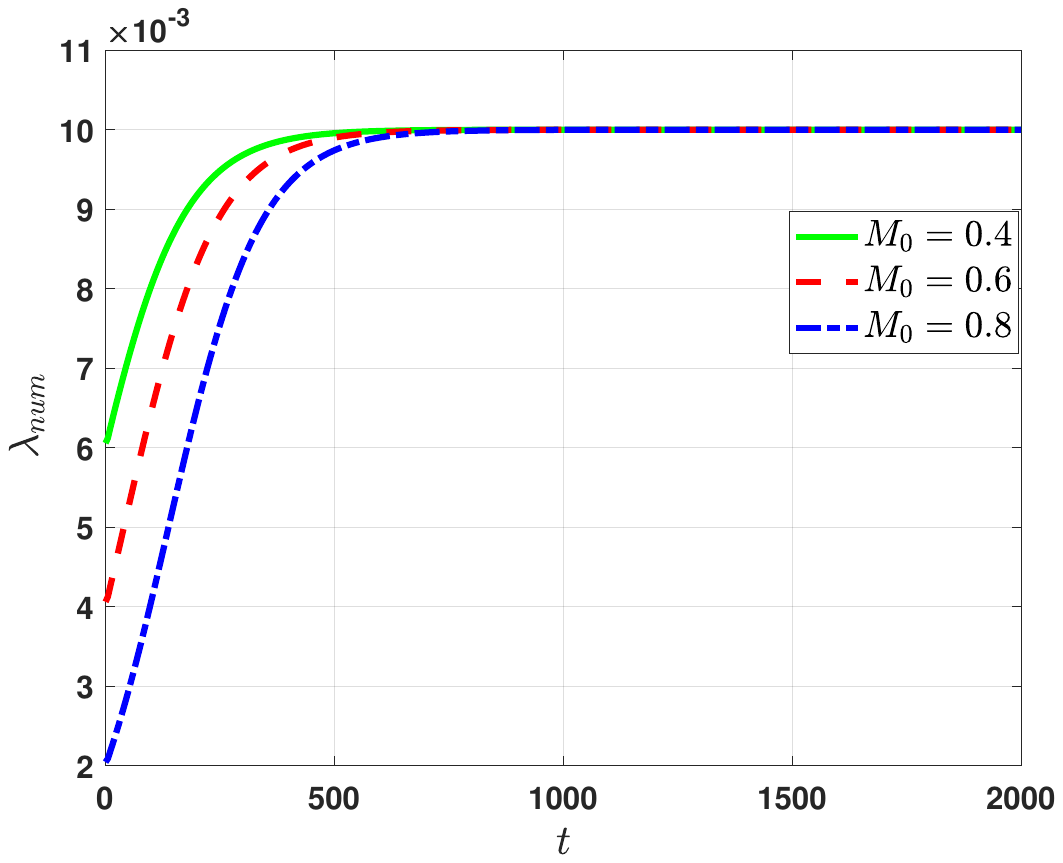}
        \caption{Numerical decay rate $\lambda_{\mathrm{num}}$ of the $L^1(\Omega)$-norm of the concentration, $\|c(t)\|_{L^1(\Omega)}$, for fixed $\kappa = 0.01$  and $M_0 = 0.4,\,0.6,\,0.8$.}
        \label{fig: decay rate vs M_0}
    \end{minipage}
    \hfill
    \begin{minipage}[t]{0.48\linewidth}
        \centering
        \includegraphics[width=1\linewidth, trim={0 210 10 210},clip]{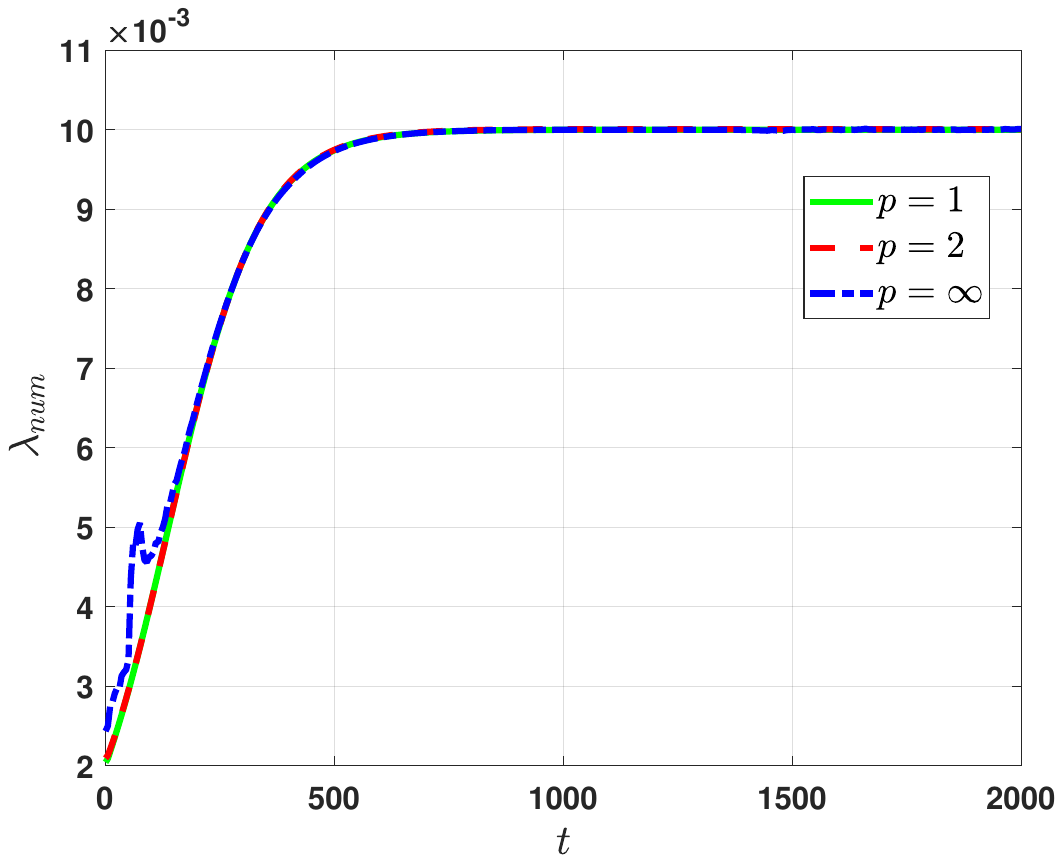}
        \caption{Numerical decay rate $\lambda_{\mathrm{num}}$ of the $L^p(\Omega)$-norm of the concentration for fixed $M_0 = 0.8$, $\kappa = 0.01$, and $p = 1,2,\infty$.}
        \label{fig: decay rate for p=1,2}
    \end{minipage}
\end{figure}
Figure \ref{fig: decay rate vs M_0} demonstrates that the numerical results substantiate the theoretical lower bound $\kappa(1 - M_0)$, as the initial decay rate $\lambda_{num}$ at $t=0$ coincides precisely with this predicted minimum for each $M_0$. Following initialization, the system evolves monotonically away from this lower bound, with all cases eventually converging to a common asymptotic limit equal to $\kappa$. The sensitivity of the numerical decay rate to the definition of the functional space is analyzed in Figure \ref{fig: decay rate for p=1,2} for $p=1, 2, \infty$. Aside from negligible transient oscillations in the $L^\infty$ norm, the curves exhibit identical growth and saturation behaviors. The results establish that the system converges uniformly to the unique steady-state limit regardless of the chosen $L^p$ norm.
\subsection{Finite-Time Blow-Up: Numerical Validation}
To validate the finite-time blowup, we choose the initial condition $c_0(\boldsymbol{x}) = 1.2$ and set $\kappa = 0.01$. Then, from Theorem \ref{thm:blowup}, we obtain the lower bound $$\|c(t)\|_{L^{\infty}} \geq \frac{6}{6 - e^{0.01\, t}}.$$This predicts a finite-time blowup at $T^* \approx 179.17$. To verify this, we plotted $\|c(t)\|_{L^{\infty}}$ versus time in Figure \ref{fig: blowup_1}. From the figure, we observe that the numerical solution (green solid line) stays strictly above the theoretical lower bound (blue dashed line) and grows explosively as $t$ approaches $T^*$. The vertical asymptote at $t = 179.17$ aligns perfectly with the theoretical prediction, confirming the blowup time.

\begin{figure}[h!]
    \centering
    \begin{minipage}[t]{0.48\linewidth}
        \centering
        \includegraphics[width=1\linewidth, trim={0 210 10 210}, clip]{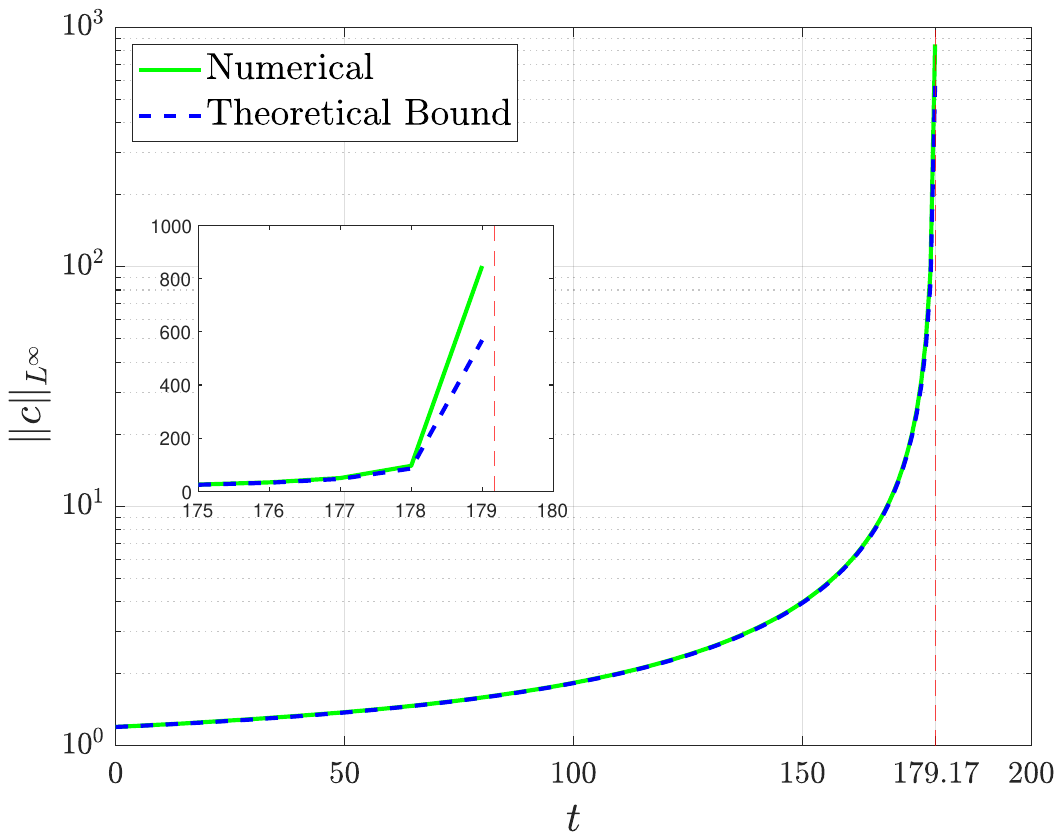}
        \caption{Time evolution of the $L^\infty(\Omega)$-norm of the concentration, $\|c(t)\|_{L^\infty(\Omega)}$, compared with the theoretical lower bound.}
        \label{fig: blowup_1}
    \end{minipage}
\end{figure}

\section{Conclusion}\label{sec: conclusion}
In this work, we analyzed a unsteady Darcy-Forchheimer-Brinkman system coupled with a convection-diffusion-reaction equation featuring concentration-dependent viscosity and a quadratic reaction term. We established local-in-time existence of weak solutions and identified a physically relevant invariant region for the concentration, which yields global existence and uniqueness of weak solutions for initial data satisfying $0 \le c_0 \le 1$. In this regime, we further proved the existence, uniqueness, and continuous dependence of strong solutions for higher regular initial data. Moreover, solutions exhibit uniform exponential decay in $L^p(\Omega)$ for all $1 \le p \le \infty$, indicating asymptotic stability of the trivial equilibrium. In contrast, when the initial concentration exceeds the unit threshold, the destabilizing effect of the superlinear reaction leads to finite-time blow-up of the $L^\infty$-norm. An explicit upper bound for the blow-up time was derived, together with a lower bound characterizing the growth mechanism responsible for the singularity.

Numerical simulations based on a finite element discretization are consistent with the analytical results, illustrating both the exponential decay behavior in the subcritical regime and the onset of finite-time blow-up for supercritical initial data. The framework developed here provides a basis for the analysis of more complex multi-component transport models in porous media involving nonlinear flow-reaction coupling.

\section*{Acknowledgments}
 S. Kundu acknowledges UGC, Government of India, for a research fellowship (Ref: 1145/CSIR-UGC NET June 2019).

\bibliographystyle{plain}
\bibliography{references}
\end{document}